\documentclass[10pt, a4paper,reqno]{amsart}
\usepackage{latexsym, amsfonts, euscript, amsmath, amssymb, amsfonts, graphicx}
\newtheorem{theorem}{Theorem}
\newtheorem{proposition}[theorem]{Proposition}
\newtheorem{lemma}[theorem]{Lemma}
\newtheorem{corollary}[theorem]{Corollary}

\newtheorem*{theorem*}{Theorem}
\newtheorem*{conjecture*}{Conjecture}
\newtheorem{definition}[theorem]{Definition}
\theoremstyle{remark}

\newcommand{\LT}{{\mathrm{LT}}}

\newcommand{\Aff}{\mathbb{A}}

\newcommand{\zrmn}{\mathcal{Z}_{r}^{m,n}}
\newcommand{\zrkmn}{\mathcal{Z}_{r,k}^{m,n}}
\newcommand{\ztwo}{\mathcal{Z}_{2}^{m,n}}
\newcommand{\ztwoone}{\mathcal{Z}_{2,2}^{m,n}}

\newcommand{\PP}{\mathbb{P}}
\newcommand{\Z}{\mathbb{Z}}
\newcommand{\Q}{\mathbb{Q}}

\newcommand{\F}{\mathbb{F}}

\newcommand{\spread}{\mathrm{sp}}
\newcommand{\Spec}{\mathrm{Spec}}
\newcommand{\supp}{\mathrm{supp}}
\newcommand{\Path}{\mathcal{P}}

\newcommand{\I}{\mathcal{I}}
\newcommand{\Inot}{\mathcal{I}_0}
\newcommand{\Dnot}{{\Delta}_0}

\begin{document}

\title{Hilbert series of certain jet schemes of determinantal varieties}

\author{Sudhir R. Ghorpade}
\address{Department of Mathematics,
Indian Institute of Technology Bombay,\newline \indent
Powai, Mumbai 400076, India.}
\email{srg@math.iitb.ac.in}

\author{Boyan Jonov}
\address{Department of Mathematics, University of California Santa Barbara,  \newline \indent
Santa Barbara, CA 93106, USA.}
\email{boyan@math.ucsb.edu}

\author{B. A. Sethuraman}
\address{Department of Mathematics, California State University Northridge,  \newline \indent
Northridge, CA 91330, USA.}
\email{al.sethuraman@csun.edu}

%\thanks{
%The first author is partially supported by Indo-Russian project INT/RFBR/P-114 from the Department of Science \& Technology, Govt. of India and  IRCC Award grant 12IRAWD009 from IIT Bombay.
%The third author
%was supported in part by NSF grants {DMS-0700904} and CCF-1318260 during the preparation of the paper.  He would like to thank the mathematics department at Indian Institute of Technology Bombay for its warm hospitality throughout the year-long visit during some of which part of this work was done. The first author would like to thank the mathematics department at California State University Northridge for the summer visit during which this collaboration took root.  The work presented in this paper originated from the Masters thesis of the second author at California State University Northridge, and is a continuation of his earlier paper \cite{boyan}.}

\date{November 6, 2013} %\today}

\begin{abstract}
We consider the affine variety $\ztwoone$ of first order jets over $\ztwo$, where $\ztwo$ is the classical
determinantal variety given by the vanishing of all $2\times 2$ minors of a generic $m\times n$ matrix.
When $2 < m \le n$, this jet scheme $\ztwoone$ has two irreducible components: a trivial component, isomorphic to an affine space, and a nontrivial component that is the closure of the jets supported over the smooth locus of $\ztwo$. This second component is referred to as the \emph{principal component} of $\ztwoone$; it is, in fact, a cone and can also be regarded as a projective subvariety of $\PP^{2mn-1}$.
We prove that the degree %multiplicity
of the principal component of $\ztwoone$ is the square of the degree %multiplicity
of $\ztwo$ and more generally, the Hilbert series  of the principal component of $\ztwoone$
 is the square of the Hilbert series  of $\ztwo$. As an application, we compute the $a$-invariant of
 the principal component of $\ztwoone$ and show that the principal component of $\ztwoone$ is Gorenstein if and only if $m=n$.
\end{abstract}

\maketitle
%\tableofcontents
%
%\marginpar{\textit{Abstract:} minor rewrite. Prin comp is the \textit{closure}; rewrote first sentence to give prominence to the jet-scheme instead of base; made explicit parameter values.}
%
\section{Introduction}

%This paper is about two miracles, a little miracle and a not-so-little miracle. To explain what these are, let us first develop some background
%and terminology.

Let $\F$ be an algebraically closed field
%
%\marginpar{\textit{Intro:} Removed first (lead-in) para.}
%
%Fix positive integers $m,n,r,k$ such that %with
%$r\le m\le n$. %$1\le r\le m\le n$.
and  $m,n,r$ be integers with $1\le r\le m\le n$.
Let $\zrmn$ denote the affine variety in $\Aff^{mn}_{\F}$ defined by
the vanishing of all $r \times r$ minors of an $m\times n$ matrix whose entries are independent indeterminates over $\F$.
Equivalently, $\zrmn$ is the locus of $m\times n$ matrices over $\F$ of rank $< r$.
This is a classical and well-studied object and a number of its properties are known. For example, we know that $\zrmn$ is irreducible,
rational, arithmetically Cohen-Macaulay and projectively normal. Moreover the multiplicity of $\zrmn$ (at its vertex, since $\zrmn$ is evidently a cone) or equivalently, the degree of the corresponding projective subvariety of $\PP^{mn-1}_{\F}$ is given by the following elegant formula
%due to Abhyankar
(cf. \cite[Rem. 20.18 and 20.19]{Ab} or \cite[Cor. 6.2]{survey}; see also %Herzog and Trung
\cite{HeTr} for an alternative proof and \cite[Ch. 2, \S 4]{ACGH}
%Ch. 2, \S 4 of \cite{ACGH} or the appendix in
or \cite[p. 352]{GK} for an alternative approach and a different formula):
\begin{equation}
\label{detmult}
e\left(\zrmn \right) = \det_{1\le i,j\le r-1}\left( {\binom{m+n-i-j}{m-i} }\right) .
\end{equation}
More generally, the Hilbert series of $\zrmn$ (or more precisely, of the corresponding projective subvariety of $\PP^{mn-1}_{\F}$) is also known
and is explicitly given by
\begin{equation}
\label{dethilb}
\displaystyle{\frac{ \sum_{k\ge 0} h_k t^k}{(1-t)^d}},
\end{equation}
where $d = (r-1)(m+n-r+1)$ is the dimension of $\zrmn$ (as an affine variety), and the coefficients $h_k$ are given by sums of binomial determinants as follows:
$$
h_k = \sum_{k_1 + \cdots + k_{r-1} = k}
 \det_{1\le i,j\le r-1} \left( { \binom{m-i}{k_i} } { \binom{n-j}{k_i+i-j} }
\right).
$$
For a proof of this formula, we refer to %can be found, for example, in
\cite{JSPI} (see also \cite{Galigo} and \cite{CH}). Using this (cf. \cite{CH}), or otherwise (cf. \cite{Svanes}), it can be shown that $\zrmn$ is Gorenstein if and only if $m=n$. Moreover, one can also %deduce a result of Gr\"abe \cite{Gr} which says
show %(cf. \cite{Gr} or \cite[Thm. 4]{JSPI})
that the $a$-invariant of the
(homogeneous) coordinate ring of $\zrmn$ [which, by definition, is the least degree of a generator of its graded canonical module] is %given by
$n(1-r)$; see, e.g., \cite{Gr} or \cite[Thm. 4]{JSPI}.
%\cite[Theorem 4]{JSPI}.

We now turn to jet schemes, which have been of much recent interest due in large part to  Nash's suggestion \cite{Nash}  that jet schemes should give information about singularities of the base; see, e.g., \cite{MM1, MM2, EM}. If ${\mathcal{Z}}$ is a scheme of finite type over $\F$ and $k$ a positive integer, then a $(k-1)$-jet on ${\mathcal{Z}}$ is a morphism $\Spec \, \F[t]/(t^k) \to {\mathcal{Z}}$. The set of
$(k-1)$-jets on ${\mathcal{Z}}$ forms a scheme of finite type over $\F$, denoted ${\mathcal{J}}_{k-1}({\mathcal{Z}})$ and called the $(k-1)^{\rm th}$ jet scheme of ${\mathcal{Z}}$. A little more concretely,
suppose ${\mathcal{Z}}$
%is an affine variety in $\Aff^N$: by this we mean that ${\mathcal{Z}}$
is the affine scheme  $\Spec \, S/I$ defined by the ideal $I=\left\langle f_1, \dots , f_s \right\rangle$ in the polynomial ring
$S=\F[X_1, \dots , X_N]$. %and if we
Consider independent indeterminates $t$ and $X_i^{({\ell})}$ ($i=1, \dots, N$ and ${\ell}=0, \dots , k-1$) over $\F$ and the corresponding polynomial
ring $S^{(k)}$ in the $Nk$ variables $X_i^{({\ell})}$.  For each $j=1, \dots , s$, the polynomial $f_j\big(X_1^{(0)}+t X_1^{(1)}+ \cdots + t^{k-1}X_1^{(k-1)}, \; \dots \; , X_N^{(0)}+tX_N^{(1)}+ \cdots + t^{k-1}X_N^{(k-1)}\big)$ is of the form $f_j^{(0)} + t f_j^{(1)} + \dots + t^{k-1}f_j^{(k-1)}$ modulo $\left\langle t^k\right\rangle$ for unique $f_j^{({\ell})}\in S^{(k)}$ ($0\le \ell < k$). %for ${\ell}=0, \dots , k-1$,
 ${\mathcal{J}}_{k-1}({\mathcal{Z}})$ is then the affine scheme $\Spec \, S^{(k)}/ I'$, where $I'$ is the ideal generated by all  $f_j^{({\ell})}$, $1\le j \le s$, $0\le \ell < k$.
 %
 %\marginpar{This remark would be helpful to explain looseness when people mix up the set of physical points making up a ``variety''  with the corresponding scheme.}
 %
 (Often in the literature, authors conflate the algebraic set  in $\Aff^{Nk}$ consisting of the zeros of the polynomials $f_j^{({\ell})}$ with ${\mathcal{J}}_{k-1}({\mathcal{Z}})$ itself. This is generally harmless, especially when considering topological properties such as components, since the points of this algebraic set correspond bijectively with
  the set of closed points of ${\mathcal{J}}_{k-1}({\mathcal{Z}})$  as $\F$ is algebraically closed,  and the set of closed points of an affine scheme is dense in the scheme.  See \cite[Ch. 2, Rem. 3.49]{Liu} for instance.)

 % affine variety in $\Aff^{Nk}$ given by the zeros of all the $f_j^{({\ell})}$.
When ${\mathcal{Z}}$ is smooth of dimension $d$, the jet scheme ${\mathcal{J}}_{k-1}({\mathcal{Z}})$ is known to be smooth of dimension $kd$.
In general,
%there is a natural projection $\pi : {\mathcal{Z}}_{k-1}({\mathcal{Z}}) \to ({\mathcal{Z}})$ and if ${\mathcal{Z}}^{\rm reg}$
%and ${\mathcal{Z}}^{\rm sing}$ are respectively the regular and the singular loci of a reduced and irreducible scheme ${\mathcal{Z}})$
%of finite type over $\F$, then the closure of ${\mathcal{J}}_{k-1}({\mathcal{Z}}^{\rm reg})$ is a principal component of  ${\mathcal{J}}_{k-1}({\mathcal{Z}})$ with dimension $kd$, while $\pi^{-1}(({\mathcal{Z}})^{\rm sing})$
${\mathcal{J}}_{k-1}({\mathcal{Z}})$ can have multiple
irreducible components, and these include a principal component that corresponds to the closure of the set of jets supported over the smooth points of the base scheme ${\mathcal{Z}}$. %; these arise from the singularities of $({\mathcal{Z}})$ and are in general,
These components are usually quite complicated and interesting.
In fact, very little seems to be known about the structure of these components and their numerical invariants such as  multiplicities.
For example, even when ${\mathcal{Z}}$ is a monomial scheme such as the one given by $X_1X_2\cdots X_e=0$, where $e\le N$, determining the irreducible components and the multiplicity of ${\mathcal{J}}_{k-1}({\mathcal{Z}})$ appears to require some effort; see, e.g., \cite{GS} and \cite{yuen1}. Irreducible components of jet schemes of toric surfaces are discussed in \cite{HM}, while the
%
%\marginpar{Added allusion to work of Sethuraman-Sivic}
%
 irreducibility of jet schemes of the commuting matrix pairs scheme is discussed in \cite{SeSi}.   In a more recent work \cite{BMS}, the Hilbert series of arc spaces (that are, in a sense, limits of $k^{\rm th}$ jet schemes as $k\to \infty$) of seemingly simple objects such as the double line $y^2=0$ are shown to have connections with the Rogers-Ramanujan identities.

Now  determinantal varieties such as $\zrmn$  above %those described in the previous paragraph
are %excellent
%
%\marginpar{Rewrote so as to replace excellent examples with classical examples, added second Kosir-Sethuraman ref.}
%classical 
natural examples of singular algebraic varieties and it is not surprising that the study of their jet schemes has been of considerable interest. This was done first by Ko{\v{s}}ir and Sethuraman in \cite{KoSe1} and \cite{KoSe2} (see also Yuen \cite{yuen2}). To describe the related results, henceforth we fix positive integers $r,k,m,n$ with $r\le m\le n$, and let $\zrkmn$ denote the $(k-1)^{\rm th}$ jet scheme on $\zrmn$. It was shown in \cite{KoSe1} that $\zrkmn$ is irreducible of codimension $k(n-m+1)$ when $r=m$, and if $r<m$, then it can have $\ge 1+ \lfloor k/2 \rfloor$ irreducible components with equality when $r=2$ or $k=2$.  A more unified result has recently been obtained by Docampo \cite{DoC} who shows that $\zrkmn$ has exactly $k+1 - \lceil k/r\rceil$ irreducible components. At any rate, the best understood case with multiple components is $\ztwoone$, where $2<m\le n$. In this case $\ztwoone = Z_0\cup Z_1$, where $Z_1$ is isomorphic to $\Aff^{mn}$ while $Z_0$ is the principal component which is the closure of the jets supported over the smooth points of the base variety $\ztwo$. Here it will be convenient to consider $2mn$ indeterminates denoted $x_{i,j}, y_{i,j}$ for $1\le i\le m$, $1\le j\le n$, and  the corresponding polynomial ring $R= \F[x_{i,j}, y_{i,j} : 1\le i\le m, \; 1\le j\le n]$.
Also let $\I = \I_{2,2}^{m,n}$  and $\Inot$ denote, respectively, the ideals of $R$ corresponding to the jet scheme $\ztwoone$
and its principal component $Z_0$.
In \cite{KoSe2}, it was shown that both $\I$ and $\Inot$ are homogeneous radical ideals of $R$ (so that $\Inot$ is prime) and moreover,
their Gr\"obner bases were explicitly determined.  The leading term ideal
%
%\marginpar{The fact that the lead term ideal is generated by squarefree monomials was already clear, Boyan's contribution lay in the analysis of the simplicial complex---rewrote slightly to reflect this.}
%
$\LT(\Inot)$ of $\Inot$ with respect to this Gr\"obner basis is generated by squarefree monomials and hence $R/\LT(\Inot)$ is the Stanley-Reisner ring of a simplicial complex $\Delta_0$.   Jonov \cite{boyan} subsequently studied this simplicial complex. He  showed that  $\Delta_0$ is shellable %so as
and thus deduced that $R/\Inot$ is Cohen-Macaulay. 
(This last result was independently obtained by Smith and Weyman as well in \cite{SW}, using their geometric technique for computing syzygies.)
Jonov also
%. Moreover, Jonov \cite{boyan} determined
found a formula for the multiplicity of $R/\Inot$, %viz.,
namely,
\begin{equation}
\label{mult}
e(R/\Inot) = {\mathop{\sum_{i=1}^m\sum_{j=1}^n}_{(i,j)\ne (m,n)}} {\binom{m+n-i-j}{m-i}} \det
\begin{pmatrix} {\binom{i+n-2}{i-1}} & {\binom{m+j-2}{m-1}} \\
{\binom{i+n-3}{i-2}} & {\binom{m+j-3}{m-2}} \end{pmatrix} .
\end{equation}

Equation (\ref{mult}) above is the starting point of the present paper. We first show that the right side of this equation simplifies remarkably to yield the pretty result
%
%\marginpar{Rewrote slightly to remove lead-in, to expand just a bit on the words ``predictable lines'' of the first draft, etc.}
%
$$
e(R/\Inot) = {\binom{m+n-2}{m-1}}^2 = e(\ztwo)^2.
$$
(this was already mentioned in \cite[Rem. 2.8]{boyan} by way of a remark).
Next, we proceed to determine the Hilbert series of $R/\Inot$ or of the principal component $Z_0$. %jet scheme $\ztwoone$.
We use the well-known connections between the Hilbert series of $R/\Inot$, that of $R/\LT(\Inot)$, and the shelling of the facets of the simplicial complex $\Delta_0$ obtained in \cite{boyan}.
With some effort we are led to an initial formula for the Hilbert series of $R/\Inot$, which is enormously complicated and involves multiple 
sums of products of binomials in the same vein as the right side of \eqref{mult}.  But we persist with the combinatorics
%
%\marginpar{I love the expression ``persist with the combinatorics!'' It captures the spirit of this enterprise perfectly!}  %
and are eventually rewarded with the main result of this paper. %, which seems to us rather like a not-so-little miracle.
Namely, just like the multiplicity, the Hilbert series of $R/\Inot$ is precisely the square of the Hilbert series of the base determinantal variety $\ztwo$. As a corollary of this, we are able to determine the $a$-invariant of $R/\Inot$ and the Hilbert series of its graded
canonical module. Moreover, we show that, as in the case of classical determinantal varieties, %the principal component
$Z_0$ is Gorenstein if and only if $m=n$.

%We end with a philosophical remark.
%We end this section with a general remark. A miracle is usually a phenomenon that appears to be astonishing, but not clearly understood. It is perhaps the same with our results concerning multiplicity and Hilbert series. The simple formulas suggest that there is a deeper algebraic truth. For example, one would obtain such formulas if the coordinate ring of $Z_0$ is the tensor product over $\F$ of the coordinate ring of $\ztwo$ with itself. However, this doesn't seem to be true. What may be feasible is that the coordinate ring of $Z_0$ is some kind of deformation of such a tensor product. However, we do not know how to prove this. The proof given here completely elementary, but highly combinatorial and rather intricate. Yet, the simplicity of final results makes one wonder if an alternative simpler proof is possible, and perhaps more importantly, if a similar result is true for higher order jet schemes of determinantal varieties or possibly some other classes of varieties. We leave this open for future investigation.

The proofs given here are completely elementary, but highly combinatorial and rather intricate.  
%To us, the underlying intuition behind the results is that at smooth points, the base variety ``looks like'' its tangent space. %plane.  
%Locally, therefore, the %tangent scheme 
%first order jet scheme ``looks like'' the product of the base with itself. Now $\ztwoone$ is the closure of the set of jets over the smooth points of the base. 
%
Heuristically, it appears to us that up to some flat deformation (such as the Gr\"obner deformation of  $\Inot$ to $\LT(\Inot)$, which preserves the Hilbert series), the coordinate ring of the 
%global object $\ztwoone$ 
principal component (suitably deformed) should look like the tensor product of the coordinate ring of the base (similarly deformed) with itself.  (This would reflect the fact that at the smooth points, the base variety locally looks like its tangent space.)
It would follow then that the Hilbert series of the principal component is the square of that of $\ztwo$. We {emphasize} that this is only heuristics (with all of its ever-present dangers); nevertheless,
%
%\marginpar{Rewrote the last para of the Intro so as to shorten it and remove the lead-in, and to include a bit of the intuition behind why the HS comes out to be the square of the base HS.}
%
we suspect that analogous results relating the Hilbert series of the principal component to that of the base scheme should hold more generally for all $\zrkmn$, and possibly also for jet schemes over a wider class of affine base schemes. % such as smooth schemes.  
We do not know how to prove this, and leave it open for  investigation.

%$\EuScript{Z} \quad \mathcal{Z}$

\section{Binomials and Lattice Paths}
\label{sec2}

In this section we %will collate together
collect some preliminaries concerning binomial coefficients,
% and their properties, elementary facts about
alterations of summations, and lattice paths. These will be useful in the sequel.

\subsection{Binomials} To begin with, let us recall that the binomial coefficient $\binom{s}{a}$ is defined for any integer parameters $s$, $a$ 
%as follows 
(and with the standard convention that the empty product is taken as $1$) as follows:
%
%\marginpar{Rewrote slightly. I think it is useful to point out this convention about the empty product, not all may know it---it certainly wasn't in the forefront of my consciousness!}
%
$$
\binom{s}{a} = \begin{cases} \displaystyle{\frac{s(s-1)\cdots(s-a+1)}{a!}}
 & \text{ if } a\ge 0, \\ 0 & \text{ if } a<0. \end{cases}
$$
In fact, this definition makes sense not only for any $s\in \Z$ but also for $s$ in any overring of $\Z$ and in particular, $s$ can be an indeterminate over $\Q$ in which case $\binom{s}{a}$ is a polynomial in $s$ of degree $a$, provided $a\ge 0$.  Now let $s, a\in \Z$. Note that
\begin{equation}
\label{binomzero}
\binom{s}{a} = 0 \Longleftrightarrow \text{either } a <0 \text{ or } a>s\ge 0.
\end{equation}
One has to be careful with the validity of some of the familiar identities; for example,
\begin{equation}
\label{VminusA}
\binom{s}{a} = \binom{s}{s-a}\Longleftrightarrow \text{either } s  \ge 0 \text{ or } s < a < 0,
\end{equation}
whereas some standard identities such as the Pascal triangle identity or its alternative equivalent version below are valid for arbitrary integer parameters:
\begin{equation}
\label{Pascal}
\binom{s}{a-1} + \binom{s}{a} = \binom{s+1}{a}  %\text{ or equivalently, }
\text{ and } \binom{s+a}{a} + \binom{s+a}{a+1} = \binom{s+a+1}{a+1}.
\end{equation}
The equivalence %with the altervative version
of the two identities above follows from the simple fact below, which is also valid for arbitrary integer parameters:
\begin{equation}
\label{Twisted}
\binom{s+a}{a} = (-1)^a\binom{-s-1}{a}, \quad \text{that is,} \quad
\binom{s}{a} = (-1)^a\binom{a-s-1}{a}.
\end{equation}
%Another useful fact, also valid for arbitrary integer parameters, and referred to as the Switching Lemma in \cite[\S \ 2]{Ab}, is the following.
%\begin{equation}
%\label{switching}
%\binom{s}{A} \binom{A}{B}= \binom{s}{B}\binom{s-B}{A-B} %\text{ and } %or equivalently, }
%%\binom{s+A}{A} \binom{s+A+B}{B}= \binom{s+A+B}{A+B}\binom{A+B}{A}.
%\end{equation}
%The proof is an easy exercise and is omitted.
We now record some basic facts, which are often used in later sections. Proofs are
easy and are briefly outlined for the sake of completeness.

\begin{lemma}
\label{GPL}
For any $e,s,t\in \Z$ with $s \le t$, we have
$$
\sum_{s<d \le t} \binom{d}{e} = \binom{t +1}{e+1} - \binom{s +1}{e+1}. 
$$
\end{lemma}

\begin{proof}
Induct on $t-s$,  using the %second
first identity in \eqref{Pascal} to rewrite $\binom{t +1}{e+1}$. %\marginpar{Slight rewrite.}
\end{proof}

The following result is a version of the so-called Chu-Vandermonde identity.

\begin{lemma}
\label{Chu}
For any $s,t,\alpha, \beta\in \Z$, we have
\begin{equation}
\label{Chu1}
\sum_{j\in \Z} \binom{s}{\alpha + j} \binom{t}{\beta - j} = \binom{s + t }{\alpha + \beta}
\end{equation}
and
\begin{equation}
\label{Chu2}
\sum_{j\in \Z} \binom{s + \alpha + j}{\alpha + j} \binom{t + \beta - j}{\beta - j} = \binom{s + t + \alpha + \beta + 1}{\alpha + \beta} ,
\end{equation}
where in view of \eqref{binomzero}, the summation on the left in \eqref{Chu1} as well as in \eqref{Chu2} is essentially finite
%(thanks to \eqref{binomzero})
in the sense that all except finitely many summands are zero.
\end{lemma}

\begin{proof}
Let $X$ be an indeterminate over $\Q$. %Consider the identity $(1+X)^(s+A}(1+X)^(t+B} = (1+X)^(s+t+A+B}$ and
Use the binomial theorem, namely,
$
(1+X)^d = \sum_{i=0}^{\infty} \binom{d}{i} X^i, %\quad \text{ for any } C\in \Z,
$
which is valid in the formal power series ring $\Q[[X]]$ for any $d\in \Z$,
and compare the coefficients of $X^{\alpha + \beta}$ on the two sides of the identity $(1+X)^{s}(1+X)^{t} = (1+X)^{s+t}$ to obtain \eqref{Chu1}. %Next, \eqref{Chu2} follows from \eqref{Chu1} and \eqref{Twisted}.
Now \eqref{Chu1} and \eqref{Twisted} imply \eqref{Chu2}.
\end{proof}

%\begin{lemma}
%\label{TwistedChu}
%For any $U,V,\alpha, \beta\in \Z$, we have
%$$
%\sum_{j\in \Z} \binom{U + \alpha + j}{\alpha + j} \binom{V + \beta - j}{\beta - j} = \binom{U + V + \alpha + \beta + 1}{\alpha + \beta} ,
%$$
%where the summation on the left is essentially finite (thanks to \eqref{binomzero}) in the sense that all except finitely many summands are zero.
%\end{lemma}
%
%\begin{proof}
%Let $X$ be an indeterminate over $\Q$. %Consider the identity $(1+X)^(U+A}(1+X)^(V+B} = (1+X)^(U+V+A+B}$ and
%Use the binomial theorem, namely,
%$$
%(1+X)^C = \sum_{i=0}^{\infty} \binom{C}{i} X^i %\quad \text{ for any } C\in \Z,
%$$
%which is valid in the formal power series ring $\Q[[X]]$ for any $C\in \Z$,
%and compare the coefficients of $X^{\alpha + \beta}$ on the two sides of the identity $(1+X)^{U+A}(1+X)^{V+B} = (1+X)^{U+V+A+B}$.
%\end{proof}

%Here is a useful variant of the result in Lemma \ref{Chu}.
\subsection{Alterations of Summations}
\label{subsec:alt}
As in \eqref{Chu1} and \eqref{Chu2} above, we will often deal with summations that are \emph{essentially finite}, by which we mean that the parameters in the sum range over an infinite set, but the summand is zero for all except finitely many values of parameters, and so the summation is, in fact, finite. It is, however, very useful that the parameters range freely over a seemingly infinite set so that useful alterations such as the ones listed below can be readily made. These are too obvious to be stated as lemmas and proved formally. But for ease of reference, we record below some elementary transformations of essentially finite summations. In what follows, $f:\Z^2\to \Q$ will denote a rational-valued function of two integer parameters with the property that the \emph{support} of $f$, namely, the set
$\{(s_1, s_2)\in \Z^2 : f(s_1, s_2)\ne 0\}$ is finite or more generally, it is \emph{diagonally finite}, i.e., for each $k\in \Z$, the set
$\{(s_1, s_2)\in \Z^2 : s_1 + s_2 = k \text{ and } f(s_1, s_2) \ne 0\}$ is finite. In this case, for any $\nu\in \Z$ and any $\alpha,\beta\in \Z$ such that $\alpha+\beta=\nu$, we have
\begin{equation}
\label{translation}
\sum_{s_1+s_2=k-\nu} f(s_1, s_2) = \sum_{t_1+t_2=k} f(t_1-\alpha, t_2-\beta),
\end{equation}
where writing $s_1+s_2=k-\nu$ below the first summation indicates that the sum is over all $(s_1, s_2)\in \Z^2$ satisfying $s_1+s_2=k-\nu$. A similar meaning applies for the second summation and in fact, for all such summations appearing in the sequel. Since the ``diagonal condition'' $t_1+t_2=k$ is symmetric, we also have
\begin{equation}
\label{alternating}
\sum_{t_1+t_2=k} f(t_1, t_2) = \sum_{t_1+t_2=k} f(t_2, t_1).
\end{equation}
Thus, for example, using \eqref{translation} and \eqref{alternating}, we find
$$
\sum_{t_1+t_2=k} f(t_1, t_2) = \sum_{t_1+t_2=k} f(t_2+1, t_1-1) = \sum_{t_1+t_2=k} f(t_1+1, t_2-1) .
$$
%Here, and hereafter, when we write $\sum_{t_1+t_2=k}$

\subsection{Lattice Paths}
\label{latpaths}
Let $A=(a,a')$ and $E=(e,e')$ be points %with integer coordinates, i.e., in
in the integer lattice $\Z^2$.
By a {\em lattice path} from $A$ to $E$ we mean a finite sequence
$L=(P_0, P_1, \dots , P_t)$ of points in $\Z^2$ with $P_0 =A$, $P_t =E$ and
$$
P_i - P_{i-1} = (1,0) \ {\rm or }  \ (0,1) \quad {\rm for} \ i =1, \dots ,t.
$$
The lattice path $L$ %is determined by
can and will be identified with its point set
$\{P_j : 0\le j\le t\}$; indeed, $L$ is obtained by simply arranging the elements of this set in a
lexicographic order. The point $A=P_0$ is called the \emph{initial point} of $L$ while $E=P_t$ is called the \emph{end point} of $L$.
We say that a point $P_j$ is a \emph{NE-turn} %or a \emph{North-East turn}
of the lattice path $L$ if $0<j<t$ and $P_j-P_{j-1}=(0,1)$ while $P_{j+1}-P_j=(1,0)$.  Note that
a lattice path is also determined by its NE-turns.

In more intuitive terms, a lattice path consists of vertical or horizontal
steps of length $1$, and a NE-turn is simply a North-East turn. For example,
a lattice path from $A = (1,1)$ to $E=(4,5)$ may be depicted as in Figure~\ref{onepathpdf}, and it may be noted that
the points $(1,2)$ and $(2,4)$ are its NE-turns.

%\includegraphics[scale=0.55]{OnePathFromEps}

%% Put here the picture.
\begin{figure}[t]
\centering
\includegraphics[scale=0.55]{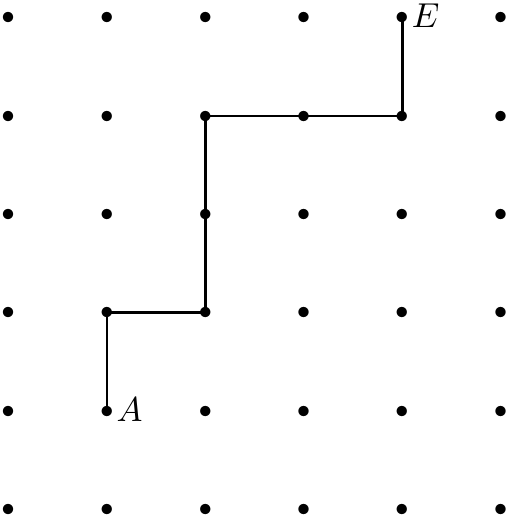}
\caption{A lattice path from $A = (1,1)$ to $E=(4,5)$} % with NE-turns at $(1,2)$ and $(2,5)$}
\label{onepathpdf}
\end{figure}

%\marginpar{\textbf{Math:} Lattice path in figure ends at $(4,5)$.}

%\noindent
%It is clear that a lattice path from $A$ to $E$ would require $e-a$ vertical steps and
%$e'-a'$ horizontal steps, and it would be determined once we choose which of the
%total $e-a+e'-a'$ (ordered) steps are vertical. Thus,
%\begin{equation}
%\label{SinglePath}
%\#\left({\rm lattice \ paths \ from \ } A \ {\rm to} \ E \right) =
%{{e-a+e'-a'}\choose{e-a}}.
%\end{equation}
%Note that this number is positive if and only if $e\ge a$ and $e'\ge a'$.
%
%Let $A, E\in \Z^2$ and $K\in \Z$.
If we let $\Path(A\to E)$ denote the set of lattice paths from $A=(a,a')$ to $E=(e,e')$ and for any $k\in \Z$, let $\Path_k(A\to E)$ denote the subset of $\Path(A\to E)$ consisting of lattice paths with exactly $k$ NE-turns, then it is easily seen that
\begin{equation}
\label{singlepath}
\left| \Path (A\to E)\right| = {\binom{e-a+e'-a'}{e-a}} %\quad
\text{ and } %\quad
\left|\Path_k(A\to E)\right| = {\binom{e-a}{k}}{\binom{e'-a'}{k}},
\end{equation}
where as usual, for a finite set $\Path$, we denote by $\left|\Path\right|$ the cardinality of $\Path$.
Given any two $d$-tuples ${\mathcal{A}} = (A_1, \dots , A_d)$ and ${\mathcal{E}} = (E_1, \dots , E_d)$ of
points in $\Z^2$, by a %$d$-{\em path}, or simply, a
{\em lattice path} from ${\mathcal{A}}$ to ${\mathcal{E}}$ we mean a $d$-tuple
${\mathcal{L}} = (L_1, \dots , L_d)$, where $L_r$ is a lattice path from $A_r$ to $E_r$,
for $1\le r\le d$. We call ${\mathcal{L}}$ to be {\em nonintersecting} if no two of the
paths $L_1, \dots , L_d$ have a point in common. We say that ${\mathcal{L}}$ has $k$ NE-turns if the total number of NE-turns in the $d$ paths $L_1, \dots , L_d$ is $k$.
%; otherwise, we call it {\em intersecting}.
%For counting the number of nonintersecting lattice paths, we have the following beautiful result.
%, due to Gessel and Viennot \cite{GV}.
The set of nonintersecting lattice paths from ${\mathcal{A}} = (A_1, \dots , A_d)$ to ${\mathcal{E}} = (E_1, \dots , E_d)$
will be denoted by $\Path\left( A_1 \to E_1, \dots , A_d \to E_d\right)$ or simply by $\Path\left(\mathcal{A} \to {\mathcal{E}}\right)$, and its subset consisting of nonintersecting lattice paths with exactly $k$ NE-turns will be denoted by $\Path_k\left( A_1 \to E_1, \dots , A_d \to E_d\right)$ or simply by $\Path_k\left(\mathcal{A} \to {\mathcal{E}}\right)$. 

\begin{proposition}
\label{GVFormula}
Let $d$ be a positive integer and let $A_r = (a_r, a'_r)$ and
$E_r = (e_r, e'_r)$, $r=1, \dots , d$, be points in $\Z^2$. %and let
Also let ${\mathcal{A}} = (A_1, \dots , A_d)$ and ${\mathcal{E}} = (E_1, \dots , E_d)$. %satisfying Assume that
\begin{enumerate}
	\item[{\rm (i)}] Suppose
%\begin{equation}
%\label{Ineqs}
$$
a_1 \le \dots \le a_d, \ \;
e_1 \le \dots \le e_d \quad {\rm and} \quad
a'_1 \ge  \dots \ge a'_d, \ \; e'_1  \ge \dots \ge e'_d   .
$$
%\end{equation*}
%\begin{eqnarray}
%&a_1 \ge a_2 \ge \dots \ge a_d \quad {\rm and} \quad  a'_1 \le a'_2 \le \dots \le a'_d;  \\
%&b_1 \ge b_2 \ge \dots \ge b_d \quad {\rm and} \quad  b'_1 \le b'_2 \le \dots \le b'_d
%\end{eqnarray}
%Let ${\mathcal{A}} = (A_1, \dots , A_d)$ and ${\mathcal{E}} = (E_1, \dots , E_d)$.
Then the number of nonintersecting lattice paths from ${\mathcal{A}}$ to ${\mathcal{E}}$ is equal to %the binomial determinant
\begin{equation}
\label{DetFormula}
\det \left( { \binom{e_j -a_i + e'_j -a'_i}{e_j - a_i} }_{1\le i, j \le d} \right)
\end{equation}
\item[{\rm (ii)}]
%Further, if we assume that
Let $k\in \Z$ and suppose
%\begin{equation*}
%\label{SharpIneqs}
$$
a_1 \le \dots \le a_d, \ \;
e_1 < \dots < e_d \quad {\rm and} \quad
a'_1 >  \dots > a'_d, \ \; e'_1  \ge \dots \ge e'_d   .
$$
%\end{equation*}
Then %for any $k\in \Z$,
the number of nonintersecting lattice paths from ${\mathcal{A}}$ to ${\mathcal{E}}$ with exactly $k$ NE-turns is equal to %the binomial determinant
\begin{equation}
\label{DetFormulaWithTurns}
%\left|\Path_k\left(\mathcal{A} \to {\mathcal{E}}\right)\right|
\sum_{k_1+ \cdots + k_d = k} \det \left( {\binom{e_j -a_i + i -j}{k_i + i - j} } {\binom{e'_j -a'_i-i+j}{k_i} }_{1\le i, j \le d} \right)
\end{equation}
\end{enumerate}
\end{proposition}

Part (i) of the above proposition is due to Gessel and Viennot \cite[Thm. 1]{GV} although some of the ideas can be traced back to
Chaundy  \cite{Chaundy}, Karlin and McGregor \cite{KM}, and Lindstr\"om \cite{Lind}. The statement here is a little more general than that of \cite{GV} and a proof can be found, for example, in \cite[\S \, 3]{Hodgenote} or \cite[\S \, 2.2]{Krat2}. Part (ii) was proved independently by Modak \cite{M}, Krattenthaler \cite{Krat} and Kulkarni \cite{Ku} (see also \cite{JSPI}), although the hypothesis in \cite{M} and \cite{Ku} on the coordinates of the initial and the end points  is slightly more restrictive than in (ii) above where we follow \cite[Thm. 1]{Krat}.
The following consequence %of part (ii) of Proposition~\ref{GVFormula}
is frequently %often
used in Section~\ref{sec:Hilb}.

\begin{corollary}
\label{corpath}
For any $a,b,c,d,s\in \Z$ with $a<c$ and $b\ge d$, %we have
the cardinality of $\Path_s\left((1,2)\to (a,b), \; (1,1)\to (c,d)\right)$ is given by
$$
%\left|\Path_s\left((1,2)\to (a,b), \; (1,1)\to (c,d)\right)\right| =
%\begin{equation}
\sum_{s_1+s_2=s} {\binom{a-1}{s_1}}{\binom{b-2}{s_1}}{\binom{c-1}{s_2}}{\binom{d-1}{s_2}}
- {\binom{a}{s_2+1}}{\binom{b-2}{s_2}}{\binom{c-2}{s_1-1}}{\binom{d-1}{s_1}}.
%\end{equation}
$$
\end{corollary}

\begin{proof}
This is just a special case of part (ii) of Proposition~\ref{GVFormula}.
\end{proof}

\section{Multiplicity}
\label{sec:mult}
%Boyan lemma for facet description, the multiplicity formula.

As in the Introduction, we fix in the remainder of this paper an algebraically closed field $\F$ and integers $m,n$ with $2<m\le n$.
Also let $x_{i,j}, \, y_{i,j}$, where $1\le i \le m$, $1\le j\le n$, be independent indeterminates over $\F$. Denote by $V_x$ the set
$\left\{x_{i,j} : \mbox{$1\le i \le m$ and $1\le j\le n$}\right\}$ of the ``$x$-variables'', and by $V_y$ a similar set of
of the ``$y$-variables''. %Also
Let $V = V_x \cup V_y$ and let $R=\F[V]$
%$R=\F\left[x_{i,j}, \, y_{i,j}\right]$
be the corresponding polynomial ring in $2mn$ variables; %, while
also, let $R_x =\F[V_x]$ and $R_y=\F[V_y]$ %, respectively,
be the corresponding polynomial rings %$\F\left[x_{i,j}\right]$ and $\F\left[y_{i,j}\right]$
in $mn$ variables. By the \emph{support} of a monomial $F$ in $R$, %(resp: $R_x$, $R_y$)
denoted $\supp(F)$, we mean the subset of $V$ %(resp: $V_x$, $V_y$)
consisting of the variables appearing in $F$. Clearly a monomial $F$ in $R$ can be uniquely written~as
\begin{equation}
\label{FxFy}
F = F_xF_y \quad \text{where $F_x,F_y$ are monomials with } F_x\in R_x \text{ and } F_y \in R_y
\end{equation}
and moreover, $F$ is squarefree if and only if both $F_x$ and $F_y$ are squarefree. Note that squarefree monomials can be identified with their supports, and in particular, faces of a simplicial complex $\Delta$ with vertex set $V$ can be viewed as squarefree monomials in $R$. 
With this in view, we may not distinguish between a squarefree monomial and its support, and we may sometimes write $x_{i,j}\in G$ rather than $x_{i,j}\mid G$ when $G$ is a squarefree monomial in $R$ and $x_{i,j}$ is a variable appearing in it. 
A monomial $G$ in $R_x$ %(resp: $R_y$)
will be called a \emph{lattice path monomial} in $R_x$ if there is a positive integer $t$ and variables $x_{i_1,j_1}, \dots , x_{i_t,j_t}$
in $V_x$ such that
\begin{equation}
\label{lpm}
G = \prod_{s=1}^t x_{i_s,j_s} \quad \text{with} \quad \left(i_s - i_{s-1}, \; j_s - j_{s-1} \right) = (1,0) \text{ or } (0,1) \text{ for } 1< s\le t.
\end{equation}
In this case $G$ is called a lattice path monomial from $x_{i_1,j_1}$ to $x_{i_t,j_t}$, and we will refer to $x_{i_1,j_1}$ as the \emph{leader} of $G$ and denote it by $\mu(G)$. Note that $\mu(G) = x_{i_1,j_1}$ depends only on $G$ (and not on the given ordering of the variables appearing in it) since $(i_1, j_1)$ is lexicographically the least among the pairs $(i,j)$ for which $x_{i,j}\in \supp(G)$.  %Moreover, 
A variable $x_{i_s,j_s}$ in $\supp(G)$ will be called an \emph{ES-turn} of $G$ if $1<s<t$, $i_s = i_{s-1}$, and $j_s = j_{s+1}$. Analgously, a variable $x_{i_s,j_s}$ in $\supp(G)$ will be called a \emph{SE-turn} of $G$ if $1<s<t$, $j_s = j_{s-1}$, and $i_s = i_{s+1}$.
%\marginpar{\textbf{Defn:} Added SE-turn, needed in Proof of Lemma 10.})
Moreover, we will call a variable $x_{i_s,j_s}$ in $\supp(G)$ the \emph{midpoint of a segment} in $G$ if $1<s<t$ and either $i_{s-1} = i_s = i_{s+1}$ (horizontal segment) or $j_{s-1} = j_s = j_{s+1}$ (vertical segment).  
It may be noted that a variable $x_{i_s,j_s}$ with $1<s<t$ is either an ES-turn or a SE-turn or the midpoint of a segment in $G$. 

Evidently, lattice path monomials in $R_x$ correspond to lattice paths in the sense of \S \ref{latpaths} if we turn the $m\times n$ rectangular matrix $\left(x_{i,j}\right)$ to the left by $90^{\circ}$ and identify the variable $x_{i,j}$ with the lattice point $(i,j)$. In this way leaders correspond to initial points while ES-turns correspond to NE-turns.
Lattice path monomials in $R_y$ together with their leaders, ES-turns, SE-turns, and midpoints of segments are similarly defined (and similarly identified with lattice paths in the sense of \S \ref{latpaths}).

We have noted in the Introduction that a Gr\"obner basis (w.r.t. reverse lexicographic order on monomials with the $2mn$ variables arranged suitably) of the ideal $\I$ of the variety $\ztwoone$ of first-order jets over $\ztwo$ as well as of
the ideal $\Inot$ of the principal component $Z_0$ of $\ztwoone$ was determined by Ko\v{s}ir and Sethuraman \cite{KoSe2}. As a consequence, one can write down the generators of the leading term ideal of $\Inot$ (cf. \cite[Prop. 1.1]{boyan}), say $\LT(\Inot)$, and deduce that
$R/\LT(\Inot)$ is the Stanley-Reisner ring of a simplicial complex $\Dnot$ with $V$ as its set of vertices. A precise description of the facets of $\Dnot$ was %obtained
given by Jonov \cite[\S \, 2]{boyan} and we recall it below.

\begin{proposition}
\label{facets} 
%\marginpar{\textbf{Math:} Reversed Sudhir's defn of which is upper Y path and which is lower.  Also, decided to uniformly write the upper path F before the lower path F---after all, the definition of the partial order gives priority to upper path over lower path.}
A squarefree monomial $F$, decomposed as in \eqref{FxFy} above, is a facet of $\Dnot$ if and only if there is a unique $(i,j)\in \Z^2$ with $1\le i\le m$, $1\le j\le n$ such that $(i,j)\ne (m,n)$ and $F_x$ is a lattice path monomial from $x_{i,j}$ to $x_{m,n}$, whereas 
$F_y= F_y^{\mathsf{U}}F_y^{\mathsf{L}}$ where
$F_y^{\mathsf{U}}$ is a lattice path monomial from $y_{1,1}$ to $y_{i,n}$,
$F_y^{\mathsf{L}}$ is a lattice path monomial from $y_{2,1}$ to $y_{m,j}$, and the supports of $F_y^{\mathsf{U}}$ and $F_y^{\mathsf{L}}$ are disjoint.
\end{proposition} 
%\marginpar{Sudhir, Would not a simple $F_y^U$ and $F_Y^L$ instead of $F_y^\mathsf{U}$ and $F_y^\mathsf{L}$ look better?  You have used the former simpler versions in Prop 9  ahead.}

The lattice path monomials $F_x$ and $F_y= F_y^{\mathsf{U}}F_y^{\mathsf{L}}$ %appearing in Proposition~\ref{facets}
are illustrated in Figure~\ref{FigFxFy} by the corresponding ``paths'' in rectangular matrices.

\begin{figure}[t]
\centering
\includegraphics[scale=0.22]{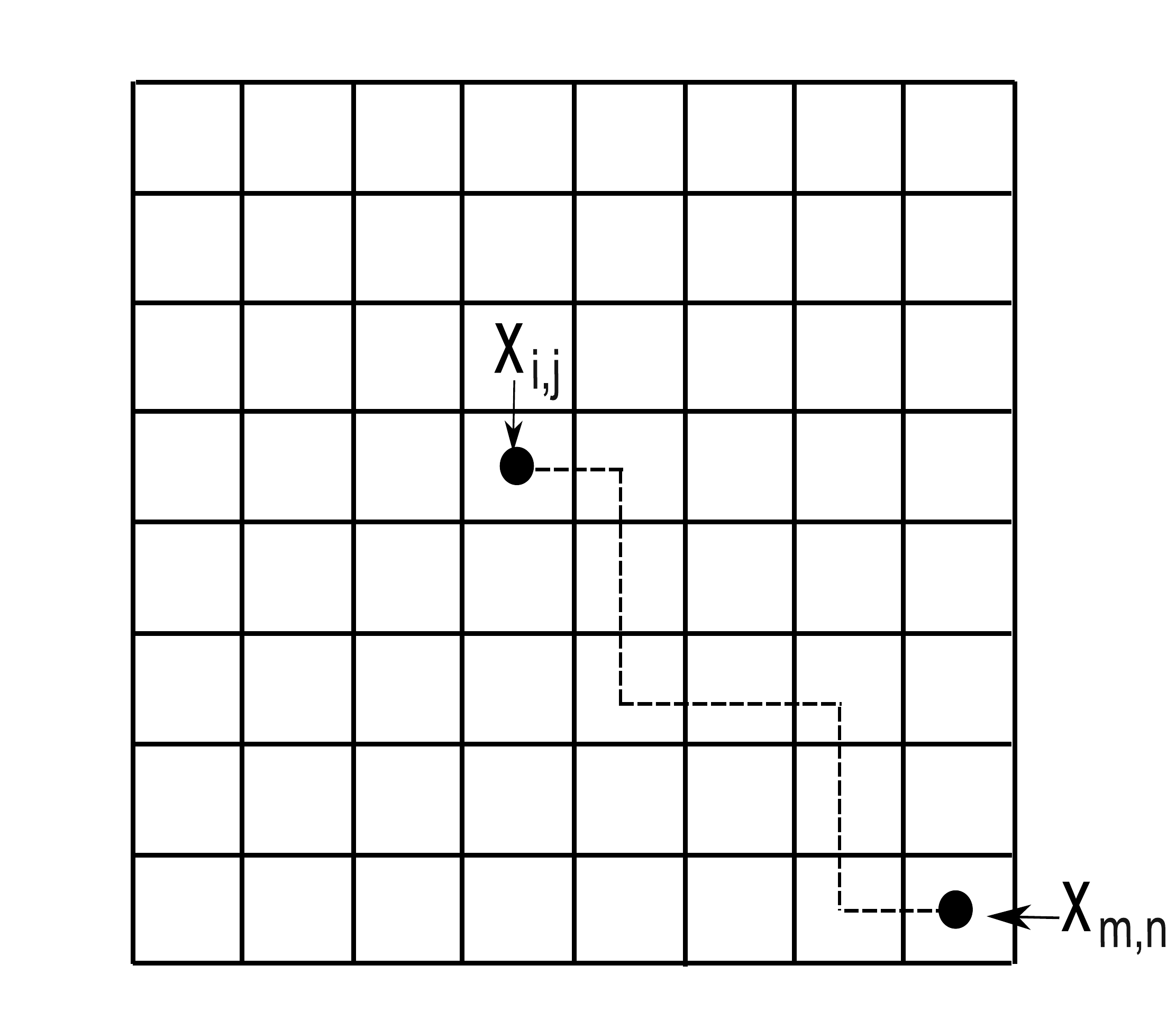}
\qquad
\includegraphics[scale=0.22]{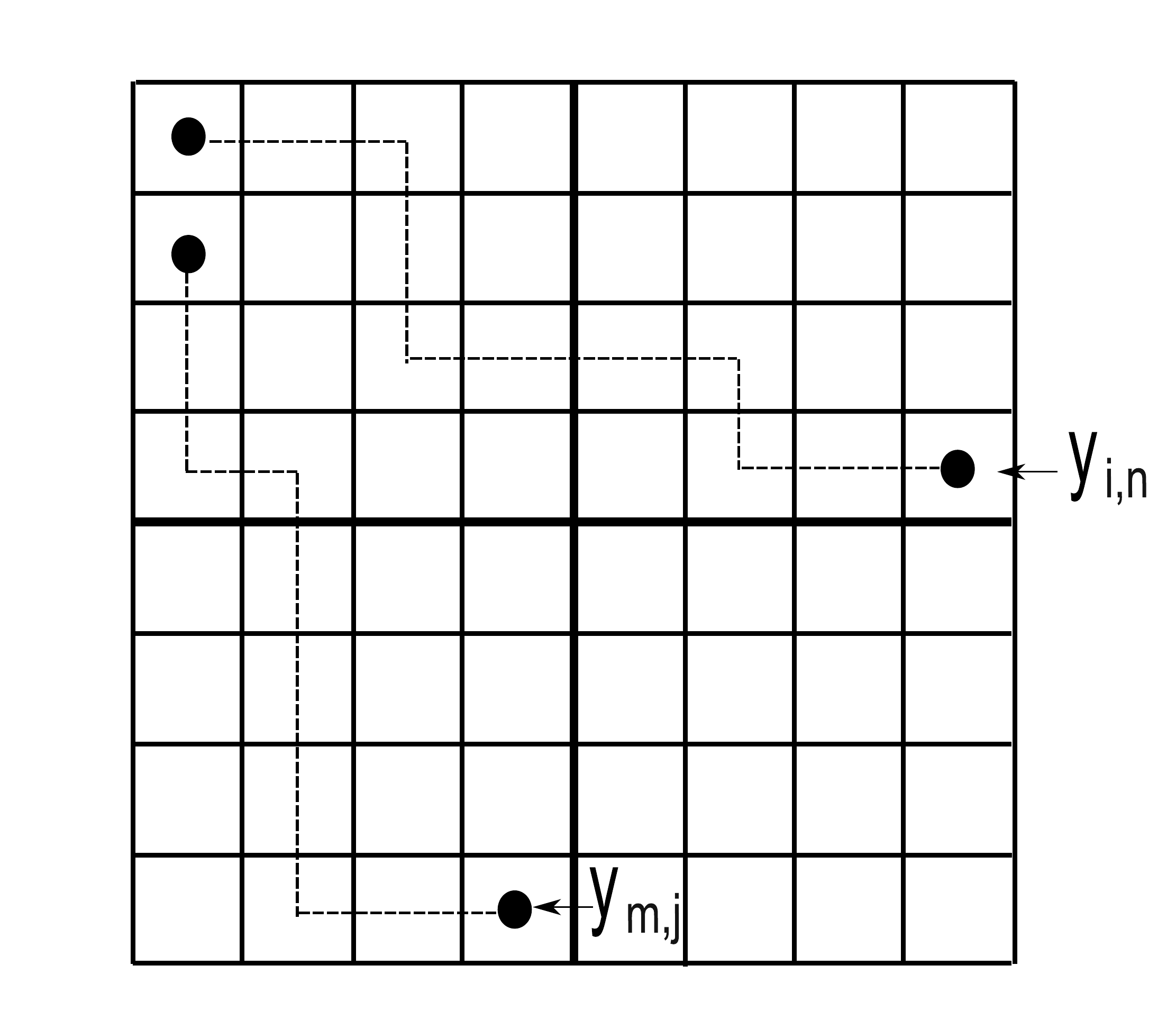}
\caption{Lattice path monomials $F_x$ and $F_y= F_y^{\mathsf{U}}F_y^{\mathsf{L}}$ in Proposition~\ref{facets}}
\label{FigFxFy}
\end{figure}

Using Proposition~\ref{facets} together with the first identity in \eqref{singlepath} and part (i) of Proposition~\ref{GVFormula}, Jonov showed that the simplicial complex $\Dnot$ is pure (i.e., all its facets have the same dimension) and deduced the %following formula
dimension and the formula stated in the Introduction for the multiplicity of the coordinate ring $R/\Inot$ of $Z_0$.

\begin{corollary}
\label{cor:jonov}
%$\dim R/\Inot =
The (Krull) dimension of $R/\Inot$ is $2(m+n-1)$ and the multiplicity of $R/\Inot$ is given by \eqref{mult}.
%\begin{equation}
%\label{mult}
%e(R/\Inot) = {\mathop{\sum_{i=1}^m\sum_{j=1}^n}_{(i,j)\ne (m,n)}} {\binom{m+n-i-j}{m-i}} \det
%\begin{pmatrix} {\binom{i+n-2}{i-1}} & {\binom{m+j-2}{m-1}} \\
%{\binom{i+n-3}{i-2}} & {\binom{m+j-3}{m-2}} \end{pmatrix}
%\end{equation}
\end{corollary}

Now here is the %little miracle 
pretty result about the multiplicity that was alluded to in the Introduction. Namely, that the %multiplicity 
formula \eqref{mult} admits a remarkable simplification.

\begin{theorem}
\label{easymult}
The multiplicity of $R/\Inot$ is given by
\begin{equation}
\label{easymultformula}
e(R/\Inot) = {\binom{m+n-2}{m-1}}^2.
\end{equation}
\end{theorem}

\begin{proof}
For $1\le i\le m$ and $1\le j \le n$, let $\Delta_{i,j}$ denote the $2\times 2$ determinant in \eqref{mult}. Observe that if $(i,j)=(m,n)$, then $\Delta_{i,j}=0$. Thus, by expanding this determinant and rearranging the summands, we can write
%$$
%e(R/\Inot) = \sum_{i=1}^m {{i+n-2}\choose{i-1}} \sum_{j=1}^n {{m+n-i-j}\choose{m-i}} {{m+j-3}\choose{m-2}}
%-  \sum_{i=1}^m {{i+n-2}\choose{i-1}} \sum_{j=1}^n {{m+n-i-j}\choose{m-i}} {{m+j-3}\choose{m-2}}
%$$
$$
e(R/\Inot) = \sum_{i=1}^m {\binom{i+n-2}{i-1}} \sum_{j=1}^n S_{i,j} %{{m+n-i-j}\choose{m-i}} {{m+j-3}\choose{m-2}}
-  \sum_{i=1}^m {\binom{i+n-3}{i-2}} \sum_{j=1}^n T_{i,j}, %{{m+n-i-j}\choose{m-i}} {{m+j-3}\choose{m-2}}
$$
where for $1\le i\le m$ and $1\le j \le n$, we have put
$$
S_{i,j} = {\binom{m+n-i-j}{m-i}} {\binom{m+j-3}{m-2}} \text{ and }
T_{i,j} = {\binom{m+n-i-j}{m-i}} {\binom{m+j-2}{m-1}} .
$$
%For each $i=1, \dots, m$,
Rewriting $S_{i,j}$ using \eqref{VminusA} and then noting that the resulting product is zero if $j<1$ or $j>n$, thanks to \eqref{binomzero},
we see from equation \eqref{Chu2} in Lemma \ref{Chu} that
$$
\sum_{j=1}^n S_{i,j} = \sum_{j} {\binom{m+n-i-j}{n-j} }{\binom{m+j-3}{j-1}} = {\binom{2m+n-i-2}{n-1} }
$$
for each $i=1, \dots, m$. In a similar manner,
$$
\sum_{j=1}^n T_{i,j} = \sum_{j} {\binom{m+n-i-j}{n-j} }{\binom{m+j-2}{j-1}} = {\binom{2m+n-i-1}{n-1} }
$$
for each $i=1, \dots, m$.
It follows that $e(R/\Inot)$ is given by the telescoping sum
$$
e(R/\Inot) = \sum_{i=1}^m (a_i - a_{i-1}) \; \text{ where } a_i:= {\binom{i+n-2}{i-1}}{\binom{2m+n-i-2}{n-1} } \ %\text{ for }
(0\le i\le m).
$$
Since $a_0=0$ and $a_m = {\binom{m+n-2}{m-1}}^2$, we obtain the desired result.
\end{proof}

It may be noted that in view of \eqref{detmult} and \eqref{easymultformula}, the multiplicity of the principal component $Z_0$
%given by \eqref{easymultformula} above
is precisely the square of the multiplicity of the base variety $\ztwo$.

\section{Hilbert series}
\label{sec:Hilb}

Let us begin by recalling that a \emph{shelling} of a pure simplicial complex $\Delta$ is a linear ordering $F_1, \dots , F_e$
of its facets such that for all positive integers $i,j$ with $j < i \le e$, there exist some $v\in F_i \setminus F_j$ and some
positive integer $k<i$ such that $F_i\setminus F_k = \{v\}$. Given such a shelling and any $t\in \{1, \dots , e\}$, we let
$$
c(F_t) = \left\{v\in F_t: \text{ there exists $s<t$ such that } F_t\setminus F_s = \{v\}\right\}.
$$
Elements of $c(F_t)$ will be referred to as the \emph{corners} of $F_t$. 	
It may be noted that $c(F_t)$ is nonempty if and only if $t>1$.
Recall also that a simplicial complex $\Delta$ is said to be \emph{shellable} if
it is pure and it has a shelling. The following result is well-known (cf. %see, e.g.,
\cite[Thm. 6.3]{BrCo}).
%It is well known that the Stanley-Reisner ring of a shellable simplicial complex is Cohen-Macaulay (see, e.g., \cite[Thm. 5.1.13]{BH}). Moreover, the Hibert series In the case of principal component $Z_0$ of

\begin{proposition}
\label{shellableHilb}
Let $\Delta$ be a shellable simplicial complex and let $R_{\Delta}$ denote its Stanley-Reisner ring. Then
\begin{enumerate}
	\item[{\rm (i)}] %If $\Delta$ is shellable, then
	$R_{\Delta}$ is Cohen-Macaulay and its (Krull) dimension $\dim R_{\Delta}$ is $1+ \dim {\Delta}$.
	\item[{\rm (ii)}] Suppose $d=\dim R_{\Delta}$ and $F_1, \dots , F_e$ is a shelling of $\Delta$. %and if
%	For $1\le t \le e$, let
%	$c(F_t) = \left\{v\in F_t: \text{ there exists $s<t$ such that } F_s\setminus F_t = \{v\}\right\}$ be the set of ``corners'' of $F_t$.
	Then the Hilbert series of $R_{\Delta}$ is given by
	$$
	\frac{\sum_{j\ge 0} h_j z^j}{(1-z)^d} \quad \text{where} \quad
	h_j = \left| \left\{t\in\{1, \dots , e\} : \left| c(F_t)\right| = j \right\} \right|\text{ for } j\ge 0.
	$$	
\end{enumerate}
\end{proposition}

In \cite{boyan}, Jonov showed that the simplicial complex $\Dnot$ mentioned in the previous section is shellable and concluded using part (i) of  Proposition~\ref{shellableHilb} that that the coordinate ring of $R/\Inot$ of the principal component $Z_0$ of $\ztwoone$ is Cohen-Macaulay. We shall now proceed to use part (ii) of Proposition \ref{shellableHilb} to determine the Hilbert series of $R/\Inot$.
%To this end, we consider the simplicial complex $\Dnot$ such that $R/\LT(\Inot)$ is the Stanley-Reisner ring of $\Dnot$. Recall
We will use the notation and terminology introduced at the beginning of Section \ref{sec:mult}. Further, we introduce the following ``anti-lexicographic'' linear order on $V_x$, % as well as on $V_y$,
i.e., on the $x$-variables. For any $x_{a,b}, x_{c,d} \in V_x$, define
% as well as the $y$-variables, and a partial order on lattice path monomials in $R_x$ as well as in $R_y$. Given any $x_{a,b}, x_{c,d}\in V_x$, we define
$$
x_{a,b} \prec x_{c,d} \Longleftrightarrow \text{ either} \quad a > c \quad \text{or} \quad a=c \, \text{ and } \, b > d.
$$
Given a lattice path monomial $G$ as in \eqref{lpm}, the \emph{spread} of $G$, denoted $\spread(G)$,  is the set of variables that are on or below the corresponding lattice path; more precisely,
$$
%\text{if } G = \prod_{s=1}^t x_{i_s,j_s} \quad \text{then} \quad
\spread(G) = \left\{x_{a,b}: i_s \le a \le m \text{ and } 1\le b \le j_s \text{ for some } s=1, \dots , t\right\}.
$$
%We introduce a %the following
%partial order on  lattice path monomials by first considering their leaders and their spreads.
%Given any lattice path monomials $G,H \in R_x$, we define
%$$
%G \le H \Longleftrightarrow \text{ either} \quad \mu(G) \prec \mu (H)  \quad \text{or} \quad \mu(G) = \mu (H) \, \text{ and }\, \spread (G) \subseteq \spread (H).
%$$
%The linear order $\prec$ and the partial order $\le$
The notion of spread is defined for lattice path monomials in $R_y$ in exactly the same manner. It may be observed that if $G,H$ are lattice path monomials (both in $R_x$ or both in $R_y$), then the condition $\spread(G) \subseteq \spread (H)$ means, roughly speaking, that $H$ is to the right of $G$; moreover, if % with the same leader, i.e., with
$\mu(G)=\mu(H)$ and  $\spread(G) = \spread (H)$, then we must have $G=H$.
%Recall that the facets of $\Dnot$ have been characterized by Proposition \ref{facets}.

Notice that the lattice path monomials $F_y^{\mathsf{U}}$ and $F_y^{\mathsf{L}}$ of Proposition \ref{facets} have the property that  $\spread(F_y^{\mathsf{L}}) \subseteq \spread(F_y^{\mathsf{U}})$. 
%\marginpar{Made explicit spread relation of upper and lower paths, need this later in Lemma 10.}

%The notion of spread of a monomial allows us to 

Following \cite{boyan}, we now define a partial order on the facets of $\Dnot$. 

\begin{definition}
\label{partord}
{\rm
For any facets $P, Q$ of $\Dnot$ with decompositions $P=P_xP_y^{\mathsf{U}}P_y^{\mathsf{L}}$
and $Q=Q_xQ_y^{\mathsf{U}}Q_y^{\mathsf{L}}$ as in Proposition \ref{facets}, define $P < Q$ if %either
one of the following four %possibilities
conditions hold:
(i) $\mu(P_x)\prec \mu (Q_x)$,
(ii)  $\mu(P_x) = \mu (Q_x)$ and $\spread(P_x) \subsetneq \spread(Q_x)$, (iii) $P_x=Q_x$ and 
$\spread(P_y^{\mathsf{U}}) \subsetneq \spread(Q_y^{\mathsf{U}})$,
(iv) $P_x=Q_x$, $P_y^{\mathsf{U}}=Q_y^{\mathsf{U}}$, and $\spread(P_y^{\mathsf{L}}) \subsetneq \spread(Q_y^{\mathsf{L}})$. 
}
\end{definition}

The next result is a consequence of \cite[Thm. 3.2]{boyan} and its proof.

\begin{proposition}
\label{shelling}
%%Consider a partial order on the facets of $\Dnot$ defined as follows.
%For any facets $P, Q$ of $\Dnot$ with decompositions $P=P_xP_y^UP_y^L$
%and $Q=Q_xQ_y^UQ_y^L$ as in Proposition \ref{facets}, define $P < Q$ if %either
%one of the following four %possibilities
%conditions hold:
%(i) $\mu(P_x)\prec \mu (Q_x)$,
%(ii)  $\mu(P_x) = \mu (Q_x)$ and $\spread(P_x) \subsetneq \spread(Q_x)$, (iii) $P_x=Q_x$ and $\spread(P_y^U) \subsetneq \spread(Q_y^U)$,
%(iv) $P_x=Q_x$, $P_y^U=Q_y^U$, and $\spread(P_y^L) \subsetneq \spread(Q_y^L)$\marginpar{\textbf{Math:} Reversed the three containments of P and Q in items ii, iii, and iv.}. Then this defines a 
The relation $<$ in Definition \ref{partord} defines a partial order and any extension of it to a total order on the facets of $\Dnot$ gives a shelling of $\Dnot$.
\end{proposition}

%We will now relate the corners of the facets of $\Dnot$ with the leaders and the ES-turns of lattice path monomials corresponding to the facets.
%Here, and hereafter,
The terminology of ES-turns can be extended from lattice path monomials to facets of $\Dnot$ as follows.
For  any facet $F$ of $\Dnot$ having a decomposition
$F=F_xF_y^{\mathsf{U}}F_y^{\mathsf{l}}$ as in Proposition \ref{facets}, by an \emph{ES-turn} of $F$ we shall mean an ES-turn of either $F_x$ or $F_y^{\mathsf{L}}$ or $F_y^{\mathsf{U}}$.
It turns out that the corners of a facet of $\Dnot$ are essentially its ES-turns or the leader of its $x$-component. There are, however, some subtleties involved and a precise relation is given below.

\begin{lemma} 
%\marginpar{Listed upper path before lower path in first sentence as well as in item iv.}
\label{corners}
Let $F$ be a facet of $\Dnot$ and $F=F_xF_y^{\mathsf{U}}F_y^{\mathsf{L}}$ be its decomposition as in  Proposition \ref{facets}. %\eqref{FxFy}.
Also let $v\in V$ be a vertex of $\Dnot$.
%, i.e., an $x$-variable or a $y$-variable.
Then we have the following.
\begin{enumerate} %\marginpar{\textbf{Math:} Item iv, the exceptional $v$ is on the upper path, not lower path.}
	\item[{\rm (i)}] If $v\in c(F)$, then  either $v = \mu(F_x)$ or $v$ is an ES-turn of $F$.
	In particular, $x_{m,n}\not\in c(F)$ and $ y_{m,n}\not\in c(F)$.
	\item[{\rm (ii)}] If $\mu(F_x) = x_{i,j}$ with $(i,j)\ne (m,n-1)$, then $\mu(F_x) \in c(F)$. Moreover, $x_{m,n-1}\not\in c(F)$.
	\item[{\rm (iii)}] If $v$ is an ES-turn of $F_x$, then $v\in c(F)$.
	\item[{\rm (iv)}] If $v$ is an ES-turn of $F_y^{\mathsf{U}}$ or of $F_y^{\mathsf{L}}$, then $v\in c(F)$, except when $v$ is an ES-turn of $F_y^{\mathsf{U}}$ such that  $v = y_{1,2}$
 %
 %\marginpar{\textbf{Math:} Item iv. Note something critical here: $v = y_{1,2}$ is another exception that hadn't been considered!  Also, the exceptional $v$ is on the upper path, not lower path.}
 %
 or when $v$ is an ES-turn of $F_y^{\mathsf{U}}$ such that $v=y_{m-1, j+1}$ and $\mu(F_x) = x_{m,j}$ for some %positive integer \
	$j<n$. % with $1<j<n$.
%	\item[{\rm (v)}]
\end{enumerate}
\end{lemma}
\begin{proof}
%We first prove \rm (i). 
(i) Let $P=P_xP_y^{\mathsf{U}}P_y^{\mathsf{L}}$ be a facet of  $\Dnot$ such that $F\setminus P = \{v\}$ and $F > P$. 
The latter implies that one of the four possibilities in Definition~\ref{partord} must arise. 
First, suppose $\mu(P_x)\prec \mu (F_x)$. Then $\mu(F_x)$ is a vertex of $F$ that is smaller than $\mu(P_x)$ in the standard lexicographic order, and hence $\mu(F_x)\not\in P_x$;  
consequently, $v = \mu(F_x)$, and we are done. Now suppose $\mu(P_x) = \mu (F_x)$ and $\spread(P_x) \subsetneq \spread(F_x)$. Then $P_x\ne F_x$ and hence $F_x\setminus P_x = \{v\}$. Note that since $\mu (F_x)$ and $x_{m,n}$ are in $P_x$, the vertex $v$ is an ES-turn, SE-turn, or the midpoint of a segment of $F_x$. In case it is the midpoint of a segment of $F_x$, the other two vertices in that segment must be in $P_x$ and since $P_x$ is a lattice path monomial, we see that $v\in P_x$, which is a contradiction. Also if $v = x_{k,l}$ (say) is a SE-turn of $F_x$, then $x_{k-1,l}$ and $x_{k,l+1}$ must be in $F_x$ and hence in $P_x$. But then $P_x$ must contain $x_{k-1,l+1}$, which is a contradiction since $x_{k-1,l+1}\not \in \spread(F_x)$. It follows that $v$ is an ES-turn of $F_x$. Next, suppose $P_x=F_x$ and $\spread(P_y^{\mathsf{U}}) \subsetneq \spread(F_y^{\mathsf{U}})$. Then $F_y^{\mathsf{U}} \setminus P_y^{\mathsf{U}} =\{v\}$. Since $\mu(P_x) = \mu (F_x)$, in view of Proposition~\ref{facets}, we see that the initial and the terminal variables of $P_y^{\mathsf{U}}$ and $F_y^{\mathsf{U}}$ coincide and so $v$ is neither of these. Arguing as in the preceding case, we can rule out the possibilities that $v$ is a SE-turn or the midpoint of a segment of $F_y^{\mathsf{U}}$. Hence $v$ is an ES-turn of $F_y^{\mathsf{U}}$. In a similar manner, we see that if $P_x=F_x$, 
$P_y^{\mathsf{U}} = F_y^{\mathsf{U}}$, and 
$\spread(P_y^{\mathsf{L}}) \subsetneq \spread(F_y^{\mathsf{L}})$, then $v$ is a ES-turn of $F_y^{\mathsf{L}}$. Thus (i) is proved.  

\medskip
  
%Next, we prove  \rm (ii). 
(ii) Let $\mu(F_x) = x_{i,j}$ with $(i,j)\ne (m,n-1)$. Then either $x_{i,j+1} \in F_x$ or $x_{i+1,j} \in F_x$. First, suppose $x_{i,j+1} \in F_x$. 
We define a new facet $P$ as follows. Let 
 $P_x = F_x \setminus \{x_{i,j}\}$ and  $ P_y^{\mathsf{L}} = F_y^{\mathsf{L}}\cup \{y_{m,j+1}\}$.  To define $P_y^{\mathsf{U}}$, we take $P_y^{\mathsf{U}} = F_y^{\mathsf{U}}$ in the case $y_{m,j+1} \notin F_y^{\mathsf{U}}$.  If $y_{m,j+1} \in F_y^{\mathsf{U}}$, then this must mean that $i=m$, and %then note  that $j$ cannot equal $n-1$ by hypothesis.  
 hence $j< n-1$. We therefore define  $P_y^{\mathsf{U}} = (F_y^{\mathsf{U}} \setminus y_{m,j+1})y_{m-1,j+2}$.  Observe that 
 $P=P_x P_y^{\mathsf{U}}P_y^{\mathsf{L}}$ is a facet of $\Delta_0$ and since $\mu(P_x) \prec \mu(F_x)$, we have $P < F$. It follows that $\mu(F_x)\in c(F)$.  
Next, suppose  $x_{i+1,j} \in F_x$. We first assume that $(i,j) \neq (m-1,n)$. Now define a new facet $P$ as follows.  First, we let $P_x = F_x \setminus \{x_{i,j}\}$.  If $y_{i+1,n} \notin F_y^{\mathsf{L}}$, then we let  $ P_y^{\mathsf{U}} = F_y^{\mathsf{U}} \cup \{y_{i+1,n}\}$ and $P_y^{\mathsf{L}} = F_y^{\mathsf{L}}$.  If $y_{i+1,n} \in F_y^{\mathsf{L}}$, then $j$ must equal $n$.  If now $i\le m-2$, then we let  $P_y^{\mathsf{L}} = F_y^{\mathsf{L}} \setminus \{y_{i+1,n}\} \cup \{y_{i+2,n-1}\}$. We are left with  the special case $i = m-1$, $j = n$. Here we let $P_x = \{x_{m,n-1},\ x_{m,n}\}$,  $P_y^{\mathsf{U}} = F_y^{\mathsf{U}} \cup \{y_{m,n}\}$, and $P_y^{\mathsf{L}} = F_y^{\mathsf{L}} \setminus \{y_{m,n}\}$. In all three cases, it is easy to verify that $P=P_x P_y^{\mathsf{U}}P_y^{\mathsf{L}}$ is a facet of $\Delta_0$ such that $F \setminus P = \{x_{i,j}\}$ and $P < F$. Consequently, $\mu(F_x)\in c(F)$.  
Finally, we show that $x_{m,n-1} \notin c(F)$. Assume, on the contrary, that there is a facet $P$  of $\Dnot$ such that 
$F \setminus P = \{x_{m,n-1}\}$.  By (i) above, $\mu(F) = x_{m,n-1}$, because there can be no ES-turn at $x_{m,n-1}$.  
In view of Proposition~\ref{facets}, 
$P$ must contain at least one variable other than $x_{m,n}$,  and since $x_{m,n-1}\not\in P$, it follows that $x_{m-1, n} \in P$.  This forces
$\mu(F_x)  \prec \mu(P_x)$, which violates the fact that $P < F$. Thus (ii) is proved. 

\medskip

(iii) Let $v = x_{k,l}$ be an ES-turn of $F_x$. Define  $P_x = F_x \setminus \{x_{k,l}\} \cup \{x_{k+1,l-1}\}$ and $P_y = F_y$. 
%,  $P_y^{\mathsf{L}} = F_y^{\mathsf{L}}$, and $P_y^{\mathsf{U}} = F_y^{\mathsf{U}}$. 
It is clear that $P=P_xP_y$ is a facet of $\Delta_0$ such that $P < F$ and $F \setminus P = \{v\}$. This proves (iii). 

\medskip

(iv) %Finally, for the proof of \rm (iv)
First, suppose %first that 
$v = y_{k,l}$ is an ES-turn of $F_y^{\mathsf{L}}$. Then $k<m$ and $l>1$. Define $P_x = F_x$, $P_y^{\mathsf{U}} = F_y^{\mathsf{U}}$, and $P_y^{\mathsf{L}} = F_y^{\mathsf{L}} \setminus \{y_{k,l} \} \cup \{y_{k+1,l-1}\}$.  It is easy to see that $P= P_xP_y^{\mathsf{U}}P_y^{\mathsf{L}}$ is facet of $\Delta_0$ such that $P < F$ and $F\setminus P = \{v\}$.  Next, suppose $v = y_{k,l}$ is an ES-turn of $F_y^{\mathsf{U}}$. 
Then once again $k<m$ and $l>1$. 
In case $y_{k+1,l-1}$ is not in $F_y^{\mathsf{L}}$, we define %$P$ by  
$P_x = F_x$, $P_y^{\mathsf{L}} = F_y^{\mathsf{L}}$, and $P_y^{\mathsf{U}} = F_y^{\mathsf{U}} \setminus \{y_{k,l}\} \cup \{y_{k+1,l-1}\}$, whereas 
in case   $y_{k+1,l-1}$ is  in $F_y^{\mathsf{L}}$ and also $k < m-1$
%\marginpar{Sudhir, note that around here we start to develop reasons why $y_{1,2}$ is not a corner.  See the condition  $l > 2$.} 
%
and $l > 2$, we define %$P$ by 
$P_x = F_x$, $P_y^{\mathsf{U}} = F_y^{\mathsf{U}} \setminus \{y_{k,l}\} \cup \{y_{k+1,l-1}\}$, and 
$P_y^{\mathsf{L}} = F_y^{\mathsf{L}} \setminus \{ y_{k+1,l-1}\} \cup \{y_{k+2,l-2}\}$. 
We verify that in both the cases, $P=P_x P_y^{\mathsf{U}}P_y^{\mathsf{L}}$ is a facet of $\Delta_0$ such that  $P < F$ and $F\setminus P = \{v\}$.

When $l = 2$, it is easy to see that  $v = y_{k,2}$ can be an ES-turn of
% \marginpar{Sudhir, here is the proof of new special case $v = y_{1,2}$, why it is not a corner.} %
$F_y^{\mathsf{U}}$ only when $k=1$ lest $F_y^{\mathsf{U}}$ and $F_y^{\mathsf{L}}$ intersect at $y_{k,1}$.  
We now show that $y_{1,2}$ is not a corner of $F$.   Suppose that $P=P_xP_y^{\mathsf{U}}P_y^{\mathsf{L}}$ is a facet of  $\Dnot$ such that $F\setminus P = \{v\}$, $v = y_{1,2}$ and $F > P$.  By Proposition~\ref{facets}, $P_y^{\mathsf{U}}$ must start at $y_{1,1}$ and $P_y^{\mathsf{L}}$ must start at $y_{2,1}$. For $P_y^{\mathsf{U}}$ to avoid $v = y_{1,2}$, it must be the case that $P_y^{\mathsf{U}}$ contains $y_{2,1}$. But this contradicts the fact that $P_y^{\mathsf{U}}$ and $P_y^{\mathsf{L}}$ do not intersect.

We are left with the situation where $k = m-1$ and $v= y_{k,l}$ is an ES-turn of $F_y^{\mathsf{U}}$ 
%(note that $k=m$ is not possible because there is an ES-turn at $y_{k,l}$) 
and moreover, $y_{m, l-1} \in F_y^{\mathsf{L}}$.  Now since $F_y^{\mathsf{U}}$ has an ES-turn at $y_{m-1, l}$, we see that $l>1$ and both 
$y_{m-1, l-1}$ and $y_{m, l}$ are in $F_y^{\mathsf{U}}$. In particular, $y_{m, l}\not\in F_y^{\mathsf{L}}$ and since  $y_{m, l-1} \in F_y^{\mathsf{L}}$, in view of Proposition~\ref{facets}, it follows that $F_y^{\mathsf{L}}$ ends at $y_{m,l-1}$, while $F_y^{\mathsf{U}}$ ends at 
$y_{m,n}$ and also that $\mu(F_x) = x_{m,l-1}$. Now if there were a facet $P=P_xP_y^{\mathsf{U}}P_y^{\mathsf{L}}$  of  $\Dnot$ such that $F\setminus P = \{v\}$ and $F > P$, then $P_x=F_x$ and $P_y^{\mathsf{L}} = F_y^{\mathsf{L}}$, whereas $F_y^{\mathsf{U}} \setminus P_y^{\mathsf{U}} = \{y_{m-1, l}\}$. But then $P_y^{\mathsf{U}}$ is a lattice path monomial that contains both $y_{m-1, l-1}$ and $y_{m, l}$ and does not contain 
$y_{m-1, l}$; so it must contain $y_{m, l-1}$. This is a contradiction since $y_{m, l-1} \in F_y^{\mathsf{L}} = P_y^{\mathsf{L}}$ and the 
monomials $P_y^{\mathsf{U}}$ and $P_y^{\mathsf{L}}$ have no variable in common. This completes the proof. 
\end{proof}

For any integers $i,j,k$ with $k\ge 0$, $1\le i \le m$ and $1\le j\le n$, we define
%$$
%C_{i,j}^k = \left|\{F : F \text{ is a facet of } \Dnot \text{ having exactly $k$ ES-turns that are in $c(F)$ and } \mu(F_x)= x_{i,j}\}\right|.
%$$
$C_{i,j}^k $ to be the number of facets $F=F_xF_y$ of $\Dnot$ such that $\mu(F_x)= x_{i,j}$ and $F$ has exactly $k$ ES-turns that are in $c(F)$.
Now a useful consequence of Lemma \ref{corners} is the following.

\begin{corollary}
\label{hkCij}
The Hilbert series of the coordinate ring $R/\Inot$ of the principal component $Z_0$ of $\ztwoone$ is given by
\begin{equation}
\label{hilbgen}
\frac{\sum_{k\ge 0} h_k z^k}{(1-z)^{2(m+n-1)}}, % \quad \text{where $h_0=1$ and } h_k = C_{m,n-1}^k + \sum_{(i,j)\ne (m,n-1)} C_{i,j}^{k-1} \text {for } k\ge 1.
\end{equation}
where $h_0=1$ and for $k\ge 1$,
\begin{equation}
\label{hkeqn}
h_k = C_{m,n-1}^k + \mathop{\sum_{(i,j)\ne (m,n-1)}}_{(i,j)\ne (m,n)} C_{i,j}^{k-1},
\end{equation}
where the last sum is over all pairs $(i,j)$ of integers satisfying $1\le i \le m$ and $1\le j\le n$ with $(i,j)\ne (m,n-1)$ and $(i,j)\ne (m,n)$.
\end{corollary}

\begin{proof}
It is well-known that the (Krull) dimension as well as the Hilbert series of $R/\Inot$ coincides with that of $R/\LT(\Inot)$ (see, e.g., \cite[\S 3]{BrCo}), where $\LT(\Inot)$ denotes the leading term ideal of $\Inot$ as in \cite{KoSe2} and \cite[Prop. 1.1]{boyan}. Now $\Dnot$ is precisely the simplicial complex such that $R/\LT(\Inot)$ is the Stanley-Reisner ring of $\Dnot$. Thus, it follows from Corollary \ref{cor:jonov} and part (ii) of Proposition \ref{shellableHilb} that the Hilbert series of $R/\Inot$ is given by \eqref{hilbgen}, where $h_0=1$ and for $k\ge 1$,
$$
h_k = \left| \left\{ F : F \text{ a facet of $\Dnot$ with } \left| c(F)\right| = k \right\}\right|.
$$
Partitioning the facets $F=F_xF_y$ in the above set in accordance with the values of $\mu(F_x)$ and noting from Proposition \ref{facets} that
$\mu(F_x)\ne (m,n)$, and then applying Lemma \ref{corners}, we obtain the desired result.
\end{proof}

%We shall now proceed to relate
We have seen in Section \ref{sec:mult} that lattice path monomials can be related to lattice paths in the sense of \S \ref{latpaths} if we rotate to the left by $90^{\circ}$ and identify the variable $x_{i,j}$ with the point $(i,j)$ of $\Z^2$. Also recall that
for any $(a,a'), (e,e') \in \Z^2$ and $s \in \Z$, we denote by $\Path_s((a,a') \to (e,e'))$ the set of lattice paths from $(a,a')$ to $(e,e')$ with $s$ NE-turns.
Likewise, if $(a_i, a'_i), (e_i, e'_i)\in \Z^2$ for $i=1, 2$ and $s \in \Z$, then by
$\Path_s\left((a_1, a'_1)\to (e_1, e'_1), \; (a_2, a'_2)\to (e_2, e'_2)\right)$
we denote the set of pairs $(L_1, L_2)$ of nonintersecting lattice paths such that $L_i$ is from  $(a_i, a'_i)$ to $(e_i, e'_i)$ for
$i=1, 2$, and the paths $L_1$ and $L_2$ together have exactly $s$ NE-turns.  Evidently, these sets are empty (and hence of cardinality $0$) when $s<0$.

\begin{figure}[t]
\centering
\includegraphics[scale=0.22]{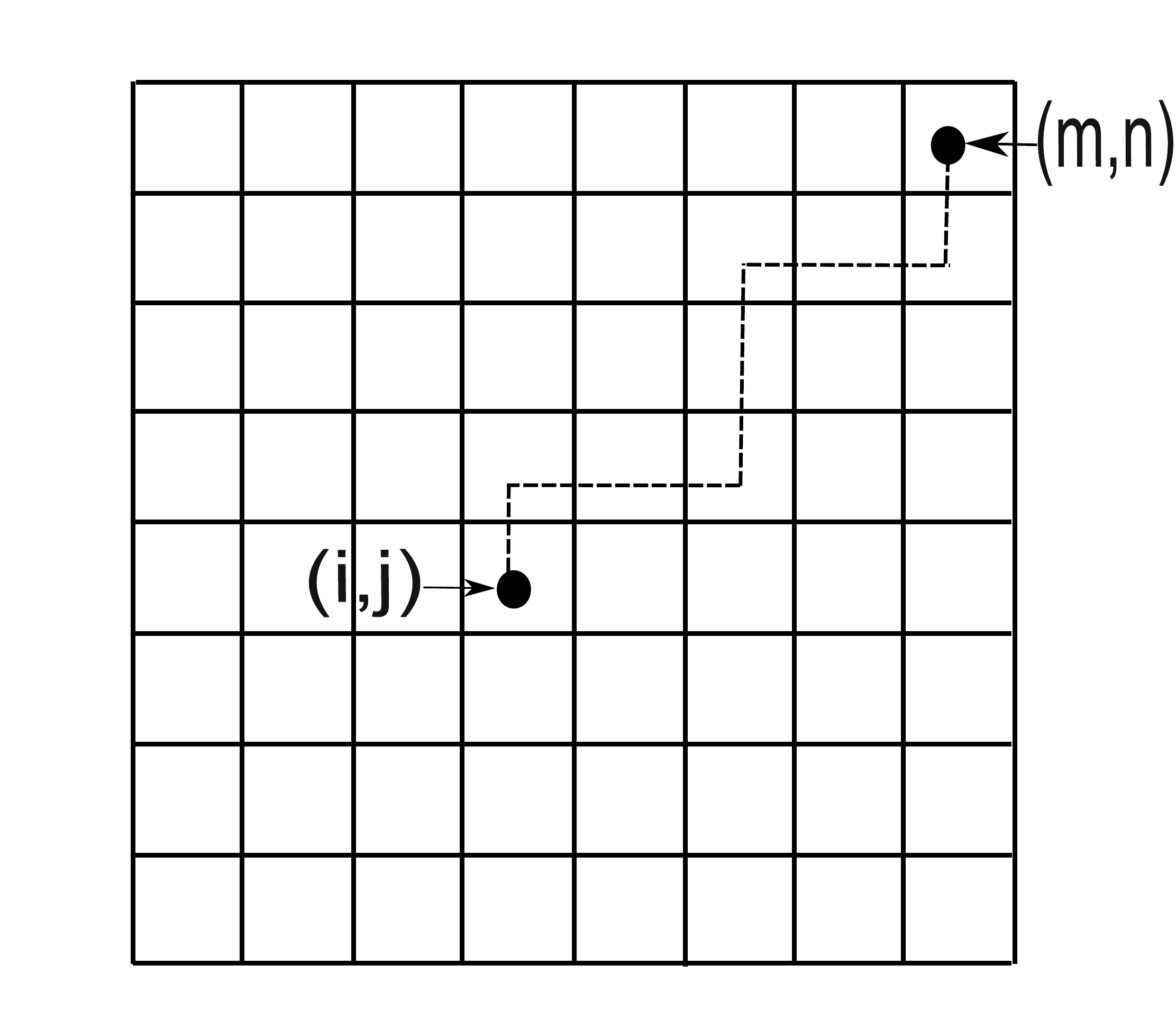}
\qquad
\includegraphics[scale=0.22]{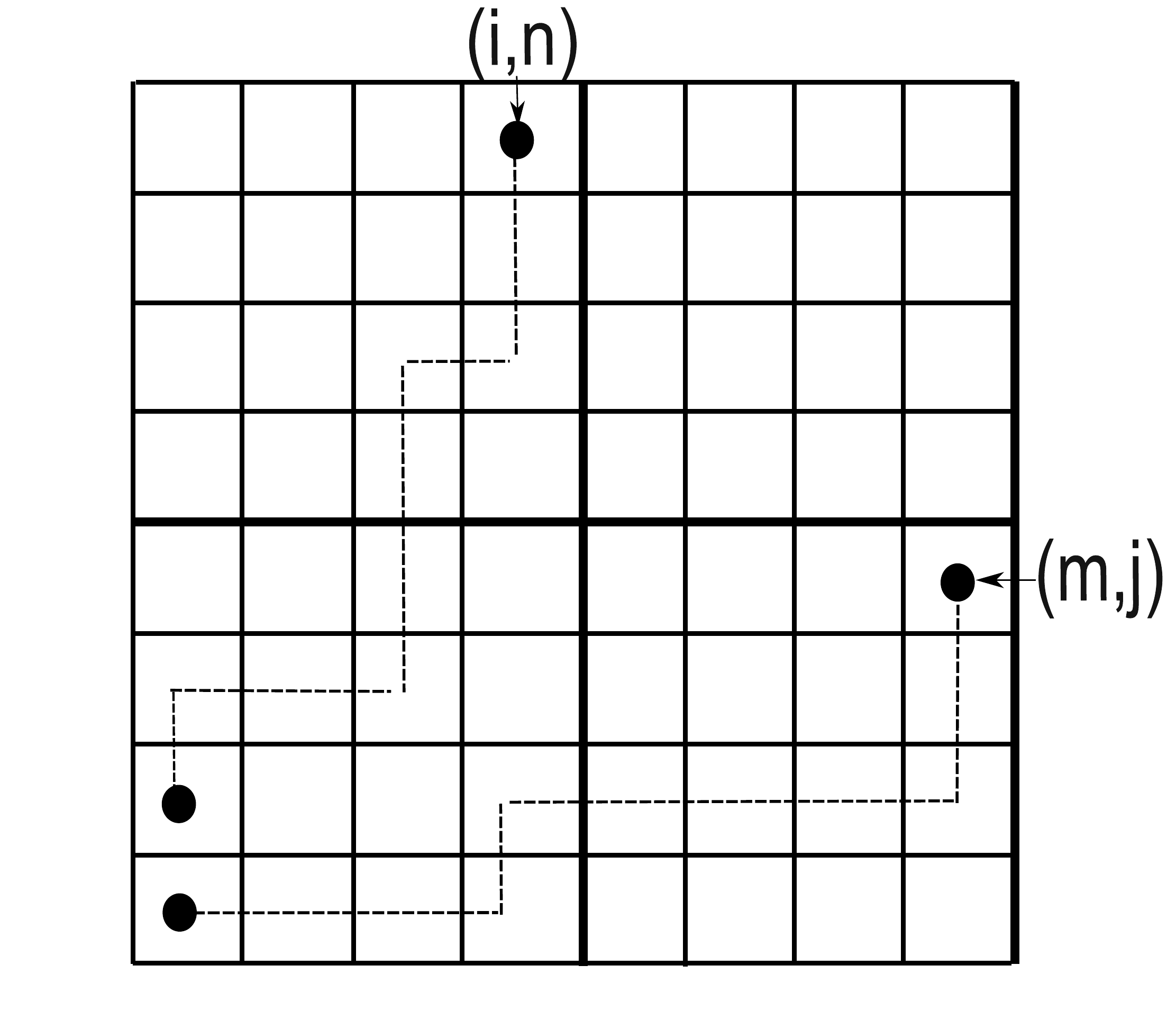}
\caption{Lattice paths $L$ and $(L_1,L_2)$ corresponding to $F_x$ and $(F_y^{\mathsf{U}}, F_y^{\mathsf{L}})$}
\label{LpathFigFxFy} 
\end{figure}

\begin{lemma}
\label{CijToPaths}
Let $s, i,j\in \Z$ with $s\ge 0$, $1\le i \le m$ and $1\le j\le n$. %Then:
\begin{enumerate}
	\item[{\rm (i)}] If $i\ne m$, then
	$$
	C_{i,j}^s = \sum_{s_1+s_2=s} \left|\Path_{s_1}\!\left((i,j)\to (m,n)\right)\right| \left| \Path_{s_2}\!\left((1,2)\to (i,n), \; (1,1)\to (m,j)\right)\right|,
	$$
	where the sum is over pairs $(s_1, s_2)$ of nonnegative integers %satisfying
	with $s_1+s_2=s$.
	\item[{\rm (ii)}] If $1< j < n-1$, then
\begin{eqnarray*}
	C_{m,j}^s &=& \sum_{p=1}^{m-1}\sum_{q=j+1}^{n-1} \left|\Path_{s-1}\left((1,2)\to (p,q), \; (1,1)\to (m,j)\right)\right| \\
 & & + \sum_{p=1}^{m-2} \left|\Path_{s-1}\left((1,2)\to (p,j), \; (1,1)\to (m,j)\right) \right|\\
 & & +  \left|\Path_{s}\left((1,2)\to (m-1,j), \; (1,1)\to (m,j)\right)\right|. \\
	\end{eqnarray*}
	\item[{\rm (iii)}]
	$C_{m,1}^s = {\binom{n-2}{s}}{\binom{m-1}{s}}$ and
	\begin{eqnarray*}
	C_{m,n-1}^s & = &  \sum_{p=1}^{m-2} \left|\Path_{s-1}\left((1,2)\to (p,n-1), \; (1,1)\to (m,n-1)\right) \right| \\
	& & + \left|\Path_{s}\left((1,2)\to (m-1,n-1), \; (1,1)\to (m,n-1)\right) \right|.
	\end{eqnarray*}
%	\item[{\rm (iv)}]
\end{enumerate}
\end{lemma}

\begin{proof}
%Let $F$ be a facet of $\Dnot$ with a decomposition $F=F_xF_y^{\mathsf{L}}F_y^{\mathsf{U}}$ as in  Proposition~\ref{facets} and let $\mu(F_x) = x_{i,j}$. As noted earlier,  $(i,j)\ne (m,n)$ and $F_x$ corresponds to a lattice path $L=L(F_x)$ %in $\Z^2$
%from $(i,j)$ to $(m,n)$, while $F_y=F_y^{\mathsf{L}}F_y^{\mathsf{U}}$ corresponds to a pair $\Ls = \Ls(F_y) = (L^*_1, L^*_2)$ of nonintersecting lattice paths, where $L^*_1$ is from $(1,1)$ to $(i,n)$ and $L^*_2$ is from $(1,2)$ to $(i,n)$. Since $2\le m\le n$ and $(i,j)\ne (m,n)$ and since $L^*1, L^*_2$ are intersecting, it follows that $(2,1)\in L^*_1$.
%
Let $i,j\in \Z$ with $1\le i \le m$, $1\le j\le n$, and $(i,j)\ne (m,n)$.
By a $90^{\circ}$ rotation to the left,  we see from Proposition~\ref{facets} that the facets $F=F_xF_y$ of $\Dnot$ with $\mu(F_x) = x_{i,j}$
are in one-to-one correspondence with the triples $(L, L^*_1, L^*_2)$ of lattice paths, where $L$ is from $(i,j)$ to $(m,n)$, while $L^*_1$ is from $(1,1)$ to $(i,n)$ and $L^*_2$ is from $(2,1)$ to 
%
 %\marginpar{\textbf{Math:} $L^*_2$ ends at $(m,j)$ instead of $(i,n)$}
 %
$(m,j)$, and moreover $L^*_1, L^*_2$ are nonintersecting. We will now modify $L^*_1, L^*_2$
slightly keeping in mind the hypothesis in Corollary~\ref{corpath}. %part (ii) of Proposition \ref{GVFormula}. %Further,
To this end, first note that $(1,2)\in L^*_1$
%
 %\marginpar{\textbf{Math:} Didn't see what $(i,j)\ne (m,n)$ had to do with $(1,2)$ being in $L^*_1$. Also, rewrote $2\le m\le n$ inequality $2< m\le n$ to be consistent with global hypothesis of the paper.}
 %
 since $2< m\le n$. % and $(i,j)\ne (m,n)$.
Thus if we let $L_1:= L^*_1\setminus\{(1,1)\}$ and $L_2:=L^*_2 \cup\{(1,1)\}$, then
$(L^*_1, L^*_2)$ and $(L_1, L_2)$ are pairs of nonintersecting lattice paths that determine each other and have exactly the same NE-turns, except 
%
 %\marginpar{\textbf{Math:} Note that if $y_{1,2}$ was a NE turn in $L^*_1$, it is  not a NE turn anymore in $L_1$.}
 %
that if $L^*_1$ had a NE turn at $(1,2)$, then $L_1$ will not have a NE turn at $(1,2$). Note though, that by Lemma \ref{corners} (iv), $y_{1,2}$ is not a corner of any facet, and this switch will therefore not affect the count of corners.
Consequently, the facets $F=F_xF_y$ of $\Dnot$ with $\mu(F_x) = x_{i,j}$ are in one-to-one correspondence with
$\Path\left((i,j)\to (m,n)\right) \times \Path\left((1,2)\to (i,n), \; (1,1)\to (m,j)\right)$.
%Also, in view of Lemma~\ref{corners}, the corners of these facets correspond either to $\mu (F_x)$ or
%to the NE-turns of the lattice paths $L, L_1, L_2$, except that in some cases
%%$\mu(F_x)$ is not to be counted as a corner and in some other cases
%a NE-turn of $L_2$ is not to be counted as a corner. Thus we consider the various cases below and do a careful book-keeping.
The lattice paths $L$ and $(L_1,L_2)$ corresponding to the components $F_x$ and $(F_y^{\mathsf{U}}, F_y^{\mathsf{L}})$ of the facet $F=F_xF_y$ are
illustrated in
 %
 %\marginpar{\textbf{Warning:} There seems to be no Figure 4??}
 %
Figure~\ref{LpathFigFxFy}; these may be compared with Figure~\ref{FigFxFy} that depicts
the lattice path monomials $F_x$ and $F_y = F_y^{\mathsf{U}} F_y^{\mathsf{L}}$.

(i) Suppose $i\ne m$. Then from Lemma~\ref{corners}, we see that %$\mu(F_x) \in c(F)$
for every facet $F=F_xF_y$ of $\Dnot$ with $\mu(F_x) = x_{i,j}$, %and also that $c(F)\setminus\{\mu(F_x)\}$ consists precisely of
all the ES-turns of 
 %
 %\marginpar{ ``$F_y^{\mathsf{L}}$ or $F_y^{\mathsf{U}}$'' interchanged in (i)}
 %
$F_x$ or $F_y^{\mathsf{U}}$ or $F_y^{\mathsf{L}}$  that are in $c(F)$
%
 %\marginpar{\textbf{Math}: Case (i) rewritten slightly to account for special case at $y_{1,2}$.}
 %
correspond to the NE-turns of the corresponding lattice paths $L$ or $L_1$ or $L_2$. From this, we readily obtain the formula in (i).

(ii) Suppose $i = m$ and $1< j < n-1$. Then for a facet $F=F_xF_y$ of $\Dnot$ with $\mu(F_x) = x_{m,j}$, the lattice path $L$ corresponding to $F_x$  is from $(m,j)$ to $(m,n)$ and evidently, this has no NE-turns. Now consider the pair $(L_1, L_2)$ in $\Path\left((1,2)\to (i,n), \; (1,1)\to (m,j)\right)$ %of lattice paths
corresponding to $\left(F_y^{\mathsf{U}}, \, F_y^{\mathsf{L}}\right)$. Suppose the last NE-turn of $L_1$ is at $(p, q+1)$.
Note that if $q<j$, then we must have $(m,j)\in L_1$, which contradicts the fact that $L_1, L_2$ are nonintersecting. Thus %we must have
$1\le p\le m-1$ and $j\le q < n$. Moreover, if $q=j$, then by part (iv) of Lemma~\ref{corners}, we see that either $p\le m-2$ or the NE-turn
$(p,q+1)$ is not in $c(F)$. It follows that $L_1$ can be replaced by %the shorter path
its truncation $\tilde L_1$, which is a lattice path from $(1,2)$ to
 %
 %\marginpar{\textbf{Math}: Truncate $L_1$ at $(p,q)$ instead at $(p,q+1)$}
 %
 $(p, q)$ such that $\tilde L_1$ and $L_2$ are nonintersecting.
Moreover, the number of NE-turns of $\tilde L_1$ in $c(F)$ are exactly one less than the number of NE-turns of $L_1$ in $c(F)$, except
when $(p,q)=(m-1, j)$ in which case they are the same. Thus by varying $(p,q)$ over an appropriate range, we obtain the formula in (ii).

(iii) If $(i,j) = (m,1)$ and $F=F_xF_y$ is a facet of $\Dnot$ with $\mu(F_x) = x_{m,1}$, then the path $L$ corresponding to $F_x$ as well as the path $L_2$ corresponding to
%
 %\marginpar{\textbf{Math}: $L_2$ corresponds to $F_y^{\mathsf{L}}$ vs $F_y^{\mathsf{U}}$}
 %
$F_y^{\mathsf{L}}$ have no NE-turns. Moreover, every NE-turn of the path $L_1\in \Path\left((1,2)\to (m,n)\right)$ corresponding to
 %
 %\marginpar{\textbf{Math}: $L_1$ corresponds to $F_y^{\mathsf{U}}$ vs $F_y^{\mathsf{L}}$}
 %
 $F_y^{\mathsf{U}}$ is necessarily in $c(F)$, thanks to Lemma~\ref{corners}. Thus, in view of \eqref{singlepath}, we see that
$C_{m,1}^s = {\binom{n-2}{s}}{\binom{m-1}{s}}$. Finally, if $(i,j)=(m,n-1)$, then arguing as in (ii) above, we see that for a facet
$F=F_xF_y$ of $\Dnot$ with $\mu(F_x) = x_{m,n-1}$, the lattice path
$L$ corresponding to $F_x$ has no NE-turns and the last NE-turn of the lattice path $L_1$ corresponding to
%
 %\marginpar{\textbf{Math}: $L_1$ corresponds to $F_y^{\mathsf{U}}$ vs $F_y^{\mathsf{L}}$}
 %
 $F_y^{\mathsf{U}}$
must be $(p, n)$ for some $p=1, \dots , m-1$. Moreover, by Lemma~\ref{corners}, this turn is counted as a corner (i.e., $x_{p,n}\in c(F)$)
if and only if $p< m-1$. Thus upon replacing $L_1$ by its truncation up to $(p, n-1)$, we obtain the desired formula for $C_{m,n-1}^s$ in (iii).
\end{proof}

We can already use the results obtained thus far to write down an explicit %, albeit complicated,
formula for the Hilbert series of the graded ring $R/\Inot$ corresponding to $Z_0$. %This follows from the case
%It may be worthwhile to record here
Indeed, it suffices to combine Corollary~\ref{hkCij},  Lemma~\ref{CijToPaths}, %, and also observe that together with
and Corollary~\ref{corpath}.
%a special case of part (ii) of Proposition~\ref{GVFormula}, %which shows %implies
%namely, that for any $a,b,c,d,s\in \Z$ with $a<c$ and $b\ge d$, %we have
%the cardinality of $\Path_s\left((1,2)\to (a,b), \; (1,1)\to (c,d)\right)$ is given by
%$$
%%\left|\Path_s\left((1,2)\to (a,b), \; (1,1)\to (c,d)\right)\right| =
%%\begin{equation}
%\sum_{s_1+s_2=s} {\binom{a-1}{s_1}}{\binom{b-2}{s_1}}{\binom{c-1}{s_2}}{\binom{d-1}{s_2}}
%- {\binom{a}{s_2+1}}{\binom{b-2}{s_2}}{\binom{c-2}{s_1-1}}{\binom{d-1}{s_1}}.
%%\end{equation}
%$$
However the resulting formula is much too complicated and we will instead use %the above formula and
%some of the
results in Section \ref{sec2} for simplifying various terms in \eqref{hkeqn} so as
to eventually arrive at an elegant formula for \eqref{hilbgen}.
%the Hilbert series.

\begin{lemma}
\label{S1S2}
Let $k$ be a positive integer. Then
%For any nonnegative integer $k$,
$C_{m,n-1}^k $ is equal to
$$
%C_{m,n-1}^k =
\sum_{t_1+t_2=k} {\binom{m-2}{t_1}}{\binom{n-2}{t_1}}{\binom{m-1}{t_2}}{\binom{n-2}{t_2}}
- {\binom{m-1}{t_2+1}}{\binom{n-2}{t_2}}{\binom{m-2}{t_1-1}}{\binom{n-2}{t_1}}.
$$
\end{lemma}

\begin{proof}
%The desired result is easily verified when $m=2$ and so we %may assume that $m>2$. %$2< m \le n$.
For %any
$s\in \Z$, let $f(s):= {\binom{m-1}{s}}{\binom{n-2}{s}}$ and $g(s):= {\binom{m-2}{s-1}}{\binom{n-2}{s}}$.
%by the special case of Proposition~\ref{GVFormula} recorded above, Then using Corollary~\ref{corpath}, we see that
%$\sum_{p=1}^{m-2} \left|\Path_{s-1}\left((1,2)\to (p,n-1), \; (1,1)\to (m,n-1)\right) \right|$ is equal to
By Corollary~\ref{corpath},
\begin{eqnarray}
\label{eq:S1}
\nonumber & & \sum_{p=1}^{m-2} \left|\Path_{k-1}\left((1,2)\to (p,n-1), \; (1,1)\to (m,n-1)\right) \right|  \\ \nonumber
&=& \sum_{p=1}^{m-2} \sum_{s_1+s_2=k-1} {\binom{p-1}{s_1}}{\binom{n-3}{s_1}} f(s_2) - {\binom{p}{s_2+1}}{\binom{n-3}{s_2}}g(s_1) \\ \nonumber
&=& %{\mathop{\sum_{s_1+s_2=k-1}}_{s_1, s_2 \ge 0}
\sum_{s_1+s_2=k-1} \left(\sum_{p'=0}^{m-3}{\binom{p'}{s_1}}\right)  {\binom{n-3}{s_1}} f(s_2)
- \left(\sum_{p=1}^{m-2}{\binom{p}{s_2+1}}\right) {\binom{n-3}{s_2}}g(s_1) \\ \nonumber
&=& \sum_{s_1+s_2=k-1} {\binom{m-2}{s_1+1}} {\binom{n-3}{s_1}}f(s_2)
- {\binom{m-1}{s_2+2}}{\binom{n-3}{s_2}}g(s_1) \\
&=& \sum_{t_1+t_2=k} {\binom{m-2}{t_1}}{\binom{n-3}{t_1-1}} f(t_2) - {\binom{m-1}{t_2+1}}{\binom{n-3}{t_2-1}} g(t_1),
\end{eqnarray}
where the penultimate equality follows from Lemma \ref{GPL} %upon noting that
since ${\binom{0}{s_1+1}} = 0 ={\binom{1}{s_2+2}}$ for $s_1, s_2\ge 0$
and also %that
since ${\binom{n-3}{s_1}}f(s_2) = 0 = {\binom{n-3}{s_2}}g(s_1)$ if %either
$s_1<0$ or $s_2<0$,
while the last equality follows  %we have put $f(s):= {\binom{m-1}{s}}{\binom{n-2}{s}}$ and $g(s):= {\binom{m-1}{s-1}}{\binom{n-2}{s}}$.
by %suitably
altering the summations (twice!) as in \eqref{translation}.
%above by the transformations $(t_1,t_2)=(s_1+1, s_2)$ or $(t_1, t_2)= (s_1, s_2-1)$, the las
On the other hand, %using
by Corollary~\ref{corpath}, %by the special case of Proposition~\ref{GVFormula} recorded above, we also see that
$
\left|\Path_{k}\left((1,2)\to (m-1,n-1), \; (1,1)\to (m,n-1)\right) \right|
$
is equal to
\begin{equation}
\label{eq:S2}
\sum_{t_1+t_2=k} {\binom{m-2}{t_1}}{\binom{n-3}{t_1}} f(t_2) - {\binom{m-1}{t_2+1}}{\binom{n-3}{t_2}}g(t_1).
\end{equation}
Now combining \eqref{eq:S1} and \eqref{eq:S2} using \eqref{Pascal}, and then using part (iii) of Lemma~\ref{CijToPaths}, we obtain the desired result.
\end{proof}

\begin{lemma}
\label{S3E3}
Let $k$ be a positive integer. Then
%For any nonnegative integer $k$,
$\sum_{i=1}^{m-1}\sum_{j=1}^n C_{i,j}^{k-1}$ is equal to
$$
\sum_{t_1+t_2=k} {\binom{m}{t_2}}{\binom{n}{t_1+1}} {\binom{m-1}{t_1}}{\binom{n-2}{t_2-1}}
- {\binom{m-1}{t_1}}{\binom{n}{t_2}}{\binom{m-1}{t_2-1}}{\binom{n-2}{t_1}}.
$$
%-  {\binom{m+1}{k - s_1+1}}{\binom{n}{k - s_2}}{\binom{m-2}{s_1-1}}{\binom{n-2}{s_2}}
%$$
%\sum_{i=1}^{m-1}\sum_{j=1}^n C_{i,j}^{k-1} =  E_3 + S_3,
%$$
%where
%$$
%E_3 = \sum_{t_1+t_2=k} {\binom{m}{t_2}}{\binom{n}{t_1+1}} {\binom{m-2}{t_1}}{\binom{n-2}{t_2-1}}
%\mathop{\sum_{k_1+s_1+s_2=k-1}}_{k_1, s_1, s_2\ge 0} {\binom{m}{k - s_2}}{\binom{n}{k - s_1}}{\binom{m-1}{s_2}}{\binom{n-2}{s_1}}
%-  {\binom{m+1}{k - s_1+1}}{\binom{n}{k - s_2}}{\binom{m-2}{s_1-1}}{\binom{n-2}{s_2}}
%$$
%and
%$$
%S_3= \sum_{t_1+t_2=k} {\binom{m}{t_2}}{\binom{n}{t_1+1}} {\binom{m-2}{t_1-1}}{\binom{n-2}{t_2-1}}
%- {\binom{m-1}{t_1}}{\binom{n}{t_2}}{\binom{m-1}{t_2-1}}{\binom{n-2}{t_1}}.
%$$
\end{lemma}

\begin{proof}
Using \eqref{singlepath} and part (i) of Lemma~\ref{CijToPaths}, we see that
$\sum_{i=1}^{m-1}\sum_{j=1}^n C_{i,j}^{k-1}$ equals % to
%\marginpar{\textbf{Math:} Check carefully. $k_1$ instead of $k_2$ in lower part of second binomial.}
$$
%\sum_{i=1}^{m-1}\sum_{j=1}^n C_{i,j}^{k-1} =
\sum_{i=1}^{m-1}\sum_{j=1}^n \sum_{k_1+k_2=k-1} {\binom{m-i}{k_1}}{\binom{n-j}{k_1}}
\left|\Path_{k_2}\left((1,2)\to (i,n), (1,1)\to (m,j)\right)\right|
$$ 
Applying Corollary~\ref{corpath} and then suitably interchanging summations and noting that the summands below are zero if $k_1<0$ or $s_1<0$ or $s_2<0$, this can be written as
\begin{equation}
\label{beast}
\mathop{\sum_{k_1+s_1+s_2=k-1}}_{k_1, s_1, s_2\ge 0} M_1 N_1 {\binom{m-1}{s_2}}{\binom{n-2}{s_1}}
- M_2 N_2 {\binom{m-2}{s_1-1}}{\binom{n-2}{s_2}},
\end{equation}
where for any given $k_1, s_1, s_2\ge 0$, we have temporarily put
\begin{eqnarray*}
& M_1 = \displaystyle{\sum_{i=1}^{m-1} {\binom{m-i}{k_1}} {\binom{i-1}{s_1}}}, \quad
& N_1 = \sum_{j=1}^{n} {\binom{n-j}{k_1}} {\binom{j-1}{s_2}} = {\binom{n}{k_1+s_2+1}}, \\
& M_2 = \displaystyle{\sum_{i=1}^{m-1} {\binom{m-i}{k_1}} {\binom{i}{s_2+1}}}, \quad
& N_2 = \sum_{j=1}^{n} {\binom{n-j}{k_1}} {\binom{j-1}{s_1}} = {\binom{n}{k_1+s_1+1}},
\end{eqnarray*}
and where the simplified expressions for $N_1, N_2$ follow by rewriting each of the summands in $N_1$ and $N_2$ using \eqref{VminusA}, invoking \eqref{binomzero} (noting that $k_1, s_1, s_2 \ge 0$),  and then applying \eqref{Chu2} for suitable values of ``$s$'', ``$t$'', ``$\alpha$'' and ``$\beta$.''
%\marginpar{Expanded explanation of simplification of $N_1$ and $N_2$.}
 A similar simplification is possible in $M_1$ and $M_2$ if we add and subtract the term corresponding to $i=m$, and in view of \eqref{binomzero}, this is only necessary if $k_1=0$. Thus, %we can write
$$
M_1 = {\binom{m}{k_1+s_1+1}} - \delta_{0,k_1}{\binom{m-1}{s_1}} \text{ and }
M_2 = {\binom{m+1}{k_1+s_2+2}} - \delta_{0,k_1}{\binom{m}{s_2+1}},
$$ 
%\marginpar{\textbf{Math:} Check carefully. Changed lower part of binomial in $M_2$ to have $s_2$ instead of $s_1$.}
where $\delta$ is the Kronecker delta. Substituting the simplified values of $M_1, N_1, M_2, N_2$ in \eqref{beast},
and letting $A(s_1, s_2):= {\binom{m-1}{s_2}}{\binom{n-2}{s_1}}$ and $B(s_1, s_2) := {\binom{m-2}{s_1-1}} {\binom{n-2}{s_2}}$
for $s_1, s_2\in \Z$, we see that \eqref{beast} is of the form $E_3+S_3$, where
$$
E_3 = \mathop{\sum_{k_1+s_1+s_2=k-1}}_{k_1, s_1, s_2\ge 0} {\binom{m}{k-s_2}}{\binom{n}{k-s_1}} A(s_1, s_2) %{\binom{m-1}{s_2}} B_{s_1} %{\binom{n-2}{s_1}}
- {\binom{m+1}{k-s_1+1}}{\binom{n}{k-s_2}} B(s_1, s_2) %{\binom{m-2}{s_1-1}} %B_{s_2} %{\binom{n-2}{s_2}}
$$
and $S_3$ is the part where the Kronecker delta is nonzero, i.e.,
$$
S_3= \sum_{s_1+s_2=k-1} {\binom{m}{s_2+1}}{\binom{n}{s_1+1}} B(s_1, s_2) %{\binom{m-2}{s_1-1}}{\binom{n-2}{s_2}}
- {\binom{m-1}{s_1}}{\binom{n}{s_2+1}} A(s_1, s_2). %{\binom{m-1}{s_2}}{\binom{n-2}{s_1}}.
$$
Altering the summation as in \eqref{translation}, we see that $S_3$ can be written as
%$$
\begin{equation}
\label{S3pure}
\sum_{t_1+t_2=k} {\binom{m}{t_2}}{\binom{n}{t_1+1}} {\binom{m-2}{t_1-1}}{\binom{n-2}{t_2-1}}
- {\binom{m-1}{t_1}}{\binom{n}{t_2}}{\binom{m-1}{t_2-1}}{\binom{n-2}{t_1}}.
\end{equation}
%$$
On the other hand, in view of \eqref{binomzero} and \eqref{alternating}, we can write
$$
E_3 = \sum_{\ell=0}^{k-1} \sum_{s_1+s_2=\ell} {\binom{m}{k-s_1}}{\binom{n}{k-s_2}} A(s_2, s_1) -
{\binom{m+1}{k-s_1+1}}{\binom{n}{k-s_2}} B(s_1, s_2).
$$
By \eqref{Pascal}, ${\binom{m+1}{k-s_1+1}} = {\binom{m}{k-s_1}} + {\binom{m}{k-(s_1-1)}}$ and using this to split the second summand %above
in $E_3$ into two parts and combining one of the parts with the first summand in $E_3$ and then applying \eqref{Pascal} once again, we see that
$$
E_3 =  \sum_{\ell=0}^{k-1} \sum_{s_1+s_2=\ell} f(s_1, s_2) - f(s_1-1, s_2),
%\quad \text{ where for } s_1, s_2\in \Z, \quad f(s_1, s_2):= {\binom{m}{k-s_1}} {\binom{n}{k-s_2}} {\binom{m-2}{s_1}} {\binom{n-2}{s_2}}
$$
where $f(s_1, s_2):= {\binom{m}{k-s_1}} {\binom{n}{k-s_2}} {\binom{m-2}{s_1}} {\binom{n-2}{s_2}}$ for $s_1, s_2\in \Z$. Now in view of \eqref{translation}, we find that $E_3$ is given by the telescoping sum
$$
E_3 = \sum_{\ell=0}^{k-1} F_{\ell} - F_{\ell-1}, \quad \text{ where for } \ell\in \Z, \quad F_{\ell}:= \sum_{s_1+s_2=\ell} f(s_1, s_2).
$$
From the definition of $f$, we see that $F_{-1}=0$, and thus $E_3 = F_{k-1}$, i.e., % or in other words,
$$
E_3= \sum_{s_1+s_2=k-1} {\binom{m}{k-s_1}} {\binom{n}{k-s_2}} {\binom{m-2}{s_1}} {\binom{n-2}{s_2}}.
$$
Now we can replace $k-s_1, k-s_2$ by $s_{2}+1, s_1+1$, respectively, % for $i=1,2$
in the above summand, and then alter the summation using \eqref{translation} to obtain
\begin{equation}
\label{E3pure}
E_3
= \sum_{t_1+t_2=k} {\binom{m}{t_2}} {\binom{n}{t_1+1}} {\binom{m-2}{t_1}} {\binom{n-2}{t_2-1}}
\end{equation}
Finally by %combining the above expression for $E_3$ with the last expression for $S_3$
adding \eqref{E3pure} and \eqref{S3pure} termwise
and using \eqref{Pascal}, we obtain the desired formula
for $E_3+S_3$, i.e., for $\sum_{i=1}^{m-1}\sum_{j=1}^n C_{i,j}^{k-1}$.
\end{proof}

\begin{lemma}
\label{S4-S7}
Let $k$ be a positive integer. Then
%For any nonnegative integer $k$,
$\sum_{j=1}^{n-2} C_{m,j}^{k-1}$ is equal to
$$
\sum_{t_1+t_2=k} {\binom{m-1}{t_1}}{\binom{n-2}{t_1}} {\binom{m-1}{t_2-1}}{\binom{n-2}{t_2}}
- {\binom{m}{t_2+1}}{\binom{n-2}{t_2}}{\binom{m-2}{t_1-2}}{\binom{n-2}{t_1}}.
$$
\end{lemma}

\begin{proof}
The desired result is easily verified when $n=3$ %2$ or $3$,
and so we %may
assume that $n>3$.
For $j, s\in \Z$, let $f_j(s):= {\binom{m-1}{s}}{\binom{j-1}{s}}$ and $g_j(s):= {\binom{m-2}{s-1}}{\binom{j-1}{s}}$.
In view of parts (iii) and (ii) of Lemma~\ref{CijToPaths} together with \eqref{binomzero} and Corollary~\ref{corpath}, we see that
%have $C_{m,1}^{k-1} = {\binom{n-2}{k-1}}{\binom{m-1}{k-1}}$ and by part (i) of the same lemma,
\begin{equation}
\label{S4-7}
C_{m,1}^{k-1} = {\binom{n-2}{k-1}}{\binom{m-1}{k-1}} \quad \text{and} \quad
\sum_{j=2}^{n-2} C_{m,j}^{k-1} = S_4 + S_5+ S_6,
\end{equation}
where %$S_4$ is equal to
\begin{eqnarray*}
S_4 \! &=& \! \sum_{j=2}^{n-2}\sum_{p=1}^{m-1}\sum_{q=j+1}^{n-1}
 \mathop{\sum_{s_1+s_2=k-2}}_{s_1,s_2\ge 0} {\binom{p-1}{s_1}}{\binom{q-2}{s_1}} f_j(s_2) - {\binom{p}{s_2+1}}{\binom{q-2}{s_2}}g_j(s_1),
\\
S_5 \! &=& \! \sum_{j=2}^{n-2}\sum_{p=1}^{m-2}
 \mathop{\sum_{s_1+s_2=k-2}}_{s_1,s_2\ge 0} {\binom{p-1}{s_1}}{\binom{j-2}{s_1}} f_j(s_2) - {\binom{p}{s_2+1}}{\binom{j-2}{s_2}}g_j(s_1),
\\
S_6 \! &= & \! \sum_{j=2}^{n-2}
 \sum_{s_1+s_2=k-1} {\binom{m-2}{s_1}}{\binom{j-2}{s_1}} f_j(s_2) - {\binom{m-1}{s_2+1}}{\binom{j-2}{s_2}}g_j(s_1).
\end{eqnarray*}
%In view of \eqref{alternating}, we can interchange $s_1$ and $s_2$ in the second summand for $S_6$ to write
Interchanging $s_1$ and $s_2$ in the second summand for $S_6$ as in \eqref{alternating}, we can write
\begin{equation}
\label{S6}
S_6 = \sum_{s_1+s_2=k-1} \lambda(s_1, s_2) \left\{{\binom{m-2}{s_1}}{\binom{m-1}{s_2}} - {\binom{m-1}{s_1+1}}{\binom{m-2}{s_2-1}}\right\},
\end{equation}
where for $s_1, s_2\in \Z$, we have let $\lambda(s_1, s_2) : = \sum_{j=2}^{n-2} {\binom{j-2}{s_1}}{\binom{j-1}{s_2}}$. Next, by Lemma~\ref{GPL},
$$
\sum_{p=1}^{m-2} {\binom{p-1}{s_1}} %= \sum_{p'=0}^{m-3} {\binom{p'}{s_1}} = {\binom{m-2}{s_1+1}} - {\binom{0}{s_1+1}}
= {\binom{m-2}{s_1+1}} \quad \text{and} \quad \sum_{p=1}^{m-2} {\binom{p}{s_2+1}} = {\binom{m-1}{s_2+2}}
\quad \text{for } s_1, s_2\ge 0.
$$ 
%\marginpar{\textbf{Math:} Check carefully. Changed last binomial above so that lower quantity is $s_2+2$. }
Consequently, by interchanging summations and rearranging terms, we find
\begin{eqnarray}
\label{S5}
S_5 \! & = & \sum_{j=2}^{n-2} \mathop{\sum_{s_1+s_2=k-2}}_{s_1,s_2\ge 0} {\binom{m-2}{s_1+1}}{\binom{j-2}{s_1}}f_j(s_2)
- {\binom{m-1}{s_2+2}}{\binom{j-2}{s_2}}g_j(s_1) \nonumber \\
\! & = & \! \!\! \sum_{s_1+s_2=k-2} \lambda(s_1, s_2) \left\{{\binom{m-2}{s_1+1}}{\binom{m-1}{s_2}} - {\binom{m-1}{s_1+2}}{\binom{m-2}{s_2-1}}\right\} \nonumber \\
\! & = & \!\! \! \sum_{s_1+s_2=k-1} \lambda(s_1-1, s_2) \left\{{\binom{m-2}{s_1}}{\binom{m-1}{s_2}} - {\binom{m-1}{s_1+1}}{\binom{m-2}{s_2-1}}\right\}
\end{eqnarray} 
%\marginpar{\textbf{Math:} Check carefully. Changed first binomial in second summand of first display line above so that lower quantity is $s_2+2$. }
where the penultimate equality follows from \eqref{binomzero} and \eqref{alternating} by interchanging $s_1$ and $s_2$ in the second summand of the preceding formula, while the last equality follows from \eqref{translation}. Now using \eqref{Pascal} we easily see that for any $s_1, s_2\in \Z$,
$$
\lambda(s_1-1, s_2) + \lambda(s_1, s_2) = \nu(s_1, s_2), \quad\text{where} \quad \nu(s_1, s_2):= \sum_{j=2}^{n-2} {\binom{j-1}{s_1}}{\binom{j-1}{s_2}}.
$$
Hence we can combine \eqref{S5} and \eqref{S6} to obtain
\begin{equation}
\label{S5S6}
S_5 + S_6 =  \sum_{s_1+s_2=k-1} \nu (s_1, s_2) \left\{{\binom{m-2}{s_1}}{\binom{m-1}{s_2}} - {\binom{m-1}{s_1+1}}{\binom{m-2}{s_2-1}}\right\}.
\end{equation}
It remains to consider $S_4$ or rather $C_{m,1}^{k-1} + S_4$. This is a little more complicated, but it can be handled using arguments similar to those in the proof of Lemma~\ref{S3E3} as follows. First, by interchanging summations and using Lemma~\ref{GPL}, we find
$$
S_4 = \sum_{j=2}^{n-2} \mathop{\sum_{s_1+s_2=k-2}}_{s_1,s_2\ge 0}  {\binom{m-1}{s_1+1}}
\theta(s_1) f_j(s_2)  - {\binom{m}{s_2+2}} \theta(s_2) g_j(s_1)
$$
where for $s\in \Z$, we have let $\theta(s):= {\binom{n-2}{s+1}} - {\binom{j-1}{s+1}}$. Now observe that
if $s_1<0$ or $s_2<0$, then $\theta(s_1) f_j(s_2) = 0 =  \theta(s_2) g_j(s_1)$.
Thus, we may drop the condition $s_1,s_2\ge 0$ in the above expression for $S_4$,
and then alter each of the two summations  over $(s_1, s_2)$ %separately
using \eqref{translation} to write
$$
S_4 = \sum_{j=2}^{n-2} \sum_{s_1+s_2=k-1} {\binom{m-1}{s_1}}
\theta(s_1-1) f_j(s_2)  - {\binom{m}{s_2+1}} \theta(s_2-1) g_j(s_1).
$$
Next, we collate the terms involving $j$ and bring the summation over $j$ inside, and note that by Lemma~\ref{GPL},
$\sum_{j=2}^{n-2} {\binom{j-1}{s}} = {\binom{n-2}{s+1}} - {\delta_{0,s}}$ for any $s \ge 0$. This yields
%$$
%\sum_{j=2}^{n-2} {\binom{j-1}{s}} = {\binom{n-2}{s+1}} - {\delta_{0,s}} \quad\text{for } s \ge 0
%$$
\begin{eqnarray*}
S_4 = \sum_{s_1+s_2=k-1} && \! {\binom{m-1}{s_1}} {\binom{n-2}{s_1}} {\binom{m-1}{s_2}}\left[{\binom{n-2}{s_2+1}} - {\delta_{0,s_2}} \right]
 \\
&&\!  -  \ {\binom{m}{s_2+1}} {\binom{n-2}{s_2}} {\binom{m-2}{s_1-1}}\left[{\binom{n-2}{s_1+1}} - {\delta_{0,s_1}} \right] \\
&& \! -  \ {\binom{m-1}{s_1}} {\binom{m-1}{s_2}} \nu(s_1, s_2) + {\binom{m}{s_2+1}} {\binom{m-2}{s_1-1}} \nu(s_1, s_2).
\end{eqnarray*}
Since ${\binom{m-2}{s_1-1}} = 0$ when $s_1=0$, the only contribution of the terms involving Kronecker delta is when $s_2=0$, and it is
$-{\binom{m-1}{k-1}} {\binom{n-2}{k-1}}$, i.e., precisely $-C_{m,1}^{k-1}$. It follows that $C_{m,1}^{k-1} + S_4 = S_4^* + E_4$, where
%$S_4^*$ is equal to 
%\marginpar{Sudhir, It'd be good to redo the display below so it reads $S_4^* = \text{etc.}$, since the quantity $S_4^*$ is explicitly refered to later.}
%$$
\begin{eqnarray*}
S_4^* & = &
 \sum_{s_1+s_2=k-1}  {\binom{m-1}{s_1}} {\binom{n-2}{s_1}} {\binom{m-1}{s_2}} {\binom{n-2}{s_2+1}} \\
& & \qquad \qquad \qquad  - \; {\binom{m}{s_2+1}} {\binom{n-2}{s_2}} {\binom{m-2}{s_1-1}} {\binom{n-2}{s_1+1}}
\end{eqnarray*}
%$$
and
\begin{eqnarray}
\label{E4}
E_4 &=& \sum_{s_1+s_2=k-1} \nu(s_1, s_2) \left\{ {\binom{m}{s_2+1}} {\binom{m-2}{s_1-1}} - {\binom{m-1}{s_1}} {\binom{m-1}{s_2}} \right\} \nonumber\\
&=& \sum_{s_1+s_2=k-1} \nu(s_1, s_2) \left\{ {\binom{m}{s_1+1}} {\binom{m-2}{s_2-1}} - {\binom{m-1}{s_1}} {\binom{m-1}{s_2}} \right\},
\end{eqnarray}
where the last equality follows by interchanging $s_1$ and $s_2$ while noting that $\nu$ is symmetric in $s_1, s_2$. 

Now combining \eqref{S5S6} and \eqref{E4}, 
%\marginpar{Break of paragraph added.} 
and then making an easy calculation using %applying
\eqref{Pascal}, we see that
$$
E_4+S_5+S_6 =  \sum_{s_1+s_2=k-1} \nu(s_1, s_2) \left\{{\binom{m-1}{s_1}}{\binom{m-2}{s_2-1}} - {\binom{m-2}{s_1-1}} {\binom{m-1}{s_2}}\right\} = 0,
$$
where the last equality follows by interchanging  $s_1$ and $s_2$ in one of the summations above.
Thus $\sum_{j=1}^{n-2} C_{m,j}^{k-1} = S_4^*$.
Finally, using \eqref{translation}, we readily see that $S_4^*$ is precisely the desired formula in the statement of the lemma.
\end{proof}

The results of Lemmas \ref{S1S2} and \ref{S4-S7} can be combined.

\begin{corollary}
\label{combn}
Let $k$ be a positive integer. Then
%For any nonnegative integer $k$,
$C_{m,n-1}^k  + \sum_{j=1}^{n-2} C_{m,j}^{k-1}$ is equal to
$$
\sum_{t_1+t_2=k} {\binom{m-1}{t_1}}{\binom{n-2}{t_1}} {\binom{m-1}{t_2}}{\binom{n-2}{t_2}}
- {\binom{m-1}{t_2+1}}{\binom{n-2}{t_2}}{\binom{m-1}{t_1-1}}{\binom{n-2}{t_1}}.
$$
\end{corollary}

\begin{proof}
Consider the formula for $\sum_{j=1}^{n-2} C_{m,j}^{k-1}$ given by Lemma~\ref{S4-S7}. This is a difference of two summations over $(t_1, t_2)\in \Z^2$ with $t_1+t_2=k$. Alter the first of these summations by interchanging $t_1$ and $t_2$,
while put
${\binom{m}{t_2+1}} = {\binom{m-1}{t_2}} + {\binom{m-1}{t_2+1}}$ in the second summation to split it into two summations.
Then using \eqref{Pascal}, we readily see that the formula %given by Lemma~\ref{S4-S7}
for $\sum_{j=1}^{n-2} C_{m,j}^{k-1}$ becomes
$$
\sum_{t_1+t_2=k} {\binom{m-2}{t_1-1}}{\binom{n-2}{t_1}} {\binom{m-1}{t_2}}{\binom{n-2}{t_2}}
- {\binom{m-1}{t_2+1}}{\binom{n-2}{t_2}}{\binom{m-2}{t_1-2}}{\binom{n-2}{t_1}}.
$$ 
%\marginpar{\textbf{Math:} Check carefully.  Changed first binomial in first summand and third binomial in second summand.}
This can be added termwise, using \eqref{Pascal} once again, with the formula for $C_{m,n-1}^k$ given by Lemma~\ref{S1S2} to obtain the desired result.
\end{proof}
We are now ready for our main theorem. %not-so-little miracle.

\begin{theorem}
\label{easyhilb}
The Hilbert series of $R/\Inot$ is given by
\begin{equation}
\label{easyHilbformula}
 \left(\frac{\sum_{e=0}^{m-1}{\binom{m-1}{e}}{\binom{n-1}{e}}z^e}{(1-z)^{m+n-1}}\right)^2.
\end{equation}
\end{theorem}

\begin{proof}
First note that \eqref{easyHilbformula} is of the form $(1-z)^{-2(m+n-1)}\sum_{k=0}^{2m-2} h_k^*z^k$, where
\begin{equation}
\label{hkP}
h^*_k= \sum_{t_1+t_2=k} {\binom{m-1}{t_1}}{\binom{n-1}{t_1}}{\binom{m-1}{t_2}}{\binom{n-1}{t_2}} \quad
\text{ for } k\in\Z.
\end{equation}
On the other hand, by Corollary~\ref{hkCij}, we see that the %numerator of the 
Hilbert series of $R/\Inot$ is given by $(1-z)^{-2(m+n-1)}\sum_{k\ge 0} h_k z^k$, where $h_0=1$ and
\begin{equation}
\label{hkeqn2}
h_k = \left(C_{m,n-1}^k + \sum_{j=1}^{n-2} C_{m,j}^{k-1} \right) + \sum_{i=1}^{m-1}\sum_{j=1}^n C_{i,j}^{k-1}  \quad \text{ for } k\ge 1.
\end{equation}
It is clear that $h^*_0=1=h_0$ and so it suffices to show that $h^*_k=h_k$ for all $k\ge 1$. In view of Corollary~\ref{combn} and Lemma~\ref{S3E3}, this is equivalent to showing that
$$
\sum_{t_1+t_2=k} P_1(t_1,t_2) - P_2(t_1,t_2) + P_3(t_1,t_2) - P_4(t_1,t_2) - P(t_1,t_2) = 0 \quad \text{ for } k\ge 1,
$$
where $P_i(t_1,t_2)$ for $i=1, \dots, 4$ and $P(t_1,t_2)$ are the relevant summands, namely,
\begin{eqnarray*}
& \textstyle{P_1(t_1,t_2) := {\binom{m-1}{t_1}}{\binom{n-2}{t_1}} {\binom{m-1}{t_2}}{\binom{n-2}{t_2}}},
& \textstyle{P_2(t_1,t_2) := {\binom{m-1}{t_2+1}}{\binom{n-2}{t_2}}{\binom{m-1}{t_1-1}}{\binom{n-2}{t_1}}}, \\
& \textstyle{P_3(t_1,t_2) := {\binom{m}{t_2}}{\binom{n}{t_1+1}} {\binom{m-1}{t_1}}{\binom{n-2}{t_2-1}}},
& \textstyle{P_4(t_1,t_2) := {\binom{m-1}{t_1}}{\binom{n}{t_2}}{\binom{m-1}{t_2-1}}{\binom{n-2}{t_1}}},
\end{eqnarray*}
and $P(t_1,t_2) := {\binom{m-1}{t_1}}{\binom{n-1}{t_1}}{\binom{m-1}{t_2}}{\binom{n-1}{t_2}}$ for $t_1, t_2\in \Z$.
%For any $t_1, t_2\in \Z$, denote the summand in \eqref{hkP} by $P(t_1,t_2)$ and the summands in
To this end, we will make an extensive use of alterations as in \eqref{translation} and \eqref{alternating}; more specifically, the fact that
$$
\sum_{t_1+t_2=k} f(t_1, t_2) = \sum_{t_1+t_2=k} f(t_2, t_1) = \sum_{t_1+t_2=k} f(t_1+1, t_2-1) = \sum_{t_1+t_2=k} f(t_2+1, t_1-1)
$$
for any $f:\Z^2\to \Q$ with finite support and any $k\in \Z$. Now fix any positive integer $k$ and any $(t_1,t_2)\in \Z^2$ with $t_1+t_2=k$. Observe that
$$
P_3(t_1-1, t_2+1) - P_4(t_2,t_1) = {\binom{m-1}{t_2+1}}{\binom{n}{t_1}} {\binom{m-1}{t_1-1}}{\binom{n-2}{t_2}} .
$$
Using (\ref{Pascal}) twice, we may substitute $ {\binom{n-2}{t_1}} + {\binom{n-2}{t_1-1}} + {\binom{n-1}{t_1-1}}$ for ${\binom{n}{t_1}}$ in the RHS of the above identity to obtain 
%\marginpar{\textbf{Math:} Added plus sign between binomials, rewrote slightly.}
$$
-P_2(t_1, t_2) + P_3(t_1-1, t_2+1) - P_4(t_2,t_1) = Q_1(t_1, t_2) + Q_2(t_1, t_2),
$$
where
$$
\textstyle{Q_1(t_1,t_2) := {\binom{m-1}{t_2+1}}{\binom{n-2}{t_1-1}} {\binom{m-1}{t_1-1}}{\binom{n-2}{t_2}}} \text{ and }
\textstyle{Q_2(t_1,t_2) := {\binom{m-1}{t_2+1}}{\binom{n-1}{t_1-1}}{\binom{m-1}{t_1-1}}{\binom{n-2}{t_2}}}.
$$
Finally, observe that $P_1(t_1, t_2) + Q_1(t_1+1, t_2-1) + Q_2(t_2+1,t_1-1) = P(t_1,t_2)$. This yields the desired result.
\end{proof} 
%\marginpar{\textbf{Math:} Changed inputs into $Q_2$, it was $t_2-1, t_1+1$. }

It may be noted that in view of \eqref{dethilb} and \eqref{easyHilbformula}, the Hilbert series of the principal component $Z_0$
%given by \eqref{easymultformula} above
is precisely the square of the Hilbert series of the base variety $\ztwo$, and as such, Theorem~\ref{easymult} could be deduced as a consequence of Theorem~\ref{easyhilb}.

As an application of Theorem~\ref{easyhilb}, we will now compute the $a$-invariant of the coordinate ring $R/\Inot$ of the principal component $Z_0$ of $\ztwoone$ and determine when $Z_0$ is Gorenstein. Recall that if $A$ is a finitely generated, positively graded Cohen-Macaulay algebra over a field, then $A$ admits a graded canonical module $\omega_A$ and the $a$-invariant of $A$ %can be
is defined as the negative of the least degree of a generator of $\omega_A$. If the Hilbert series of $A$ is given by $H_A(z)= h(z)/(1-z)^d$, where $d=\dim A$ and $h(z)\in \Q[z]$ with $h(1)\ne 0$, then
the $a$-invariant of $A$ is the order of pole of $H_A(z)$ at infinity, viz., $-\left(d-\deg h(z)\right)$. Moreover, the Hilbert series of $\omega_A$ is given by $H_{\omega_A}(z) = (-1)^d H_A(z^{-1})$.  As a general reference for these notions and results, one may consult \cite{BH}, especially Sections 3.6 and 4.4. The following result is an analogue of a theorem of Gr\"abe \cite{Gr} (see also \cite[Thm. 4]{JSPI})
for classical determinantal varieties which says that if $1\le r \le m\le n$, then the $a$-invariant of (the coordinate ring of) ${\mathcal{Z}_{r}^{m,n}}$ is $-(r-1)n$.

\begin{corollary}
\label{ainv}
The $a$-invariant of $R/\Inot$ is equal to $-2n$ and the Hilbert series of the graded canonical module of $R/\Inot$ is given by
\begin{equation}
\label{Hilbomega}
 \left(\frac{\sum_{e=0}^{m-1}{\binom{m-1}{e}}{\binom{n-1}{e}}z^{m+n-e-1}}{(1-z)^{m+n-1}}\right)^2.
\end{equation}
\end{corollary}

\begin{proof} We know from \cite[Thm. 1.2]{boyan} that $A= R/\Inot$ is Cohen-Macaulay and it is obviously a finitely generated, positively graded  $\F$-algebra. Moreover, by Theorem~\ref{easyhilb}, the Hilbert series of $A$ is given by $h_0(z)/(1-z)^{2(m+n-1)}$, where
$$
h_0(z) = \left(\sum_{e=0}^{m-1}{\binom{m-1}{e}}{\binom{n-1}{e}}z^e\right)^2.
$$
Since $2\le m\le n$, we see that %the coefficients of $h_0(z)$ are nonegative integers
$h_0(z)$ is a polynomial in $z$ of degree $2(m-1)$ with leading coefficient ${\binom{n-1}{m-1}}^2$ and all other coefficients nonnegative integers; in particular, $h_0(1)\ne 0$. Hence the $a$-invariant of $A= R/\Inot$ is $2(m-1) - 2(m+n-1) = -2n$ and also that the Hilbert series of $\omega_A$ is
given by \eqref{Hilbomega}.
\end{proof}

The following result is an analogue of a theorem of Svanes \cite{Svanes} (see also \cite{CH})
for classical determinantal varieties which says that %if $1\le r \le m\le n$, then
for any $r\ge 1$, (the coordinate ring of) ${\mathcal{Z}_{r}^{m,n}}$ is
Gorenstein if and only if $m=n$.

\begin{corollary}
\label{Gorenstein}
The coordinate ring $R/\Inot$ of $Z_0$ is Gorenstein if and only if $m=n$.
\end{corollary}

\begin{proof}
%In view of
By \cite[Thm. 1.2]{boyan} and \cite[Prop. 3.3]{KoSe2}, %the fact that $Z_0$ is irreducible, we see that
$A= R/\Inot$ is a Cohen-Macaulay domain.
Hence from a well-known result of Stanley \cite[Thm. 4.4]{Stan} (see also \cite[Cor. 4.4.6]{BH}), we see that $A$ is Gorenstein if and only if
$H_A(z) = (-1)^dz^aH_A(z^{-1})$ for some $a\in \Z$. Moreover, in this case the integer $a$ is necessarily the $a$-invariant of $A$. Thus, from
Corollary~\ref{ainv}, we see that $R/\Inot$ is Gorenstein if and only if
$$
\left[\sum_{e=0}^{m-1}{\binom{m-1}{e}}{\binom{n-1}{e}}z^e\right]^2 =
\left[\sum_{e=0}^{m-1}{\binom{m-1}{e}}{\binom{n-1}{e}}z^{m-1-e}\right]^2.
$$
Since both the polynomials inside the square brackets on the two sides of the above equality have positive leading coefficients, it follows that
$R/\Inot$ is Gorenstein if and only if ${\binom{n-1}{e}} = {\binom{n-1}{m-1-e}}$ for all $e=0,1,\dots, m-1$. Since $1< m-1 \le n-1$, the latter clearly holds if and only if $m=n$.
\end{proof}

\section*{Acknowledgments}
\begin{small}
{
The first author is partially supported by Indo-Russian project INT/RFBR/P-114 from the Department of Science \& Technology, Govt. of India and  IRCC Award grant 12IRAWD009 from IIT Bombay.
The third author
was supported in part by NSF grants {DMS-0700904} and CCF-1318260 during the preparation of the paper.  He would like to thank the mathematics department at Indian Institute of Technology Bombay for its warm hospitality throughout the year-long visit during some of which part of this work was done. The first author would like to thank the mathematics department at California State University Northridge for the summer visit during which this collaboration took root.  The work presented in this paper originated from the Masters thesis of the second author at California State University Northridge, and is a continuation of his earlier paper \cite{boyan}.}
\end{small}
%
%\label{secAck}
%\begin{small}
%We are grateful to the Otto M{\o}nsted Foundation, which supported the visit of Sudhir Ghorpade to the Technical University of Denmark during
%May-July 2010 when some of this work was carried out.
%\end{small}

\end{document}